\documentclass[]{article}

\usepackage[fleqn]{amsmath}
\usepackage{amssymb}
\usepackage{enumerate}
\usepackage{hyperref}
\usepackage{amsthm}
\usepackage{verbatim}
\usepackage[a4paper, total={13.5cm, 23.5cm}]{geometry}
\usepackage{fixfoot}
\usepackage{ centernot}
\usepackage[utf8]{inputenc}
\usepackage[T1]{fontenc}
\usepackage[parfill]{parskip}
\usepackage{graphicx}
\usepackage[dvipsnames]{xcolor}
\usepackage{bbm}
\usepackage{tabu}
\usepackage{tikz-qtree,tikz-qtree-compat}
\usepackage{todonotes}
\usepackage{csquotes}
\usepackage{enumitem}
\usepackage{tikz, tkz-euclide}
\usepackage{calc}
\usepackage[backend=bibtex8, style=authoryear-comp, doi=false, isbn=false, url=false, maxcitenames=2,maxbibnames=200000]{biblatex}
\bibliography{ref} 
\usepackage{float}
\usepackage{pgfplots}
\pgfplotsset{compat=1.14} 

\usepackage{relsize}
\usepackage{tkz-euclide}
\usepackage{mathtools}
\usepackage{mathrsfs}
\usepackage{amsthm}
\usepackage{thmtools}
\usepackage[titletoc]{appendix}
\usepackage{chngcntr}
\usepackage{abstract}
\usepackage{booktabs}
\usepackage{ltablex}
\usepackage{wrapfig}
\usepackage{accents}
\usepackage{bm}
\usepackage{ upgreek }
\usepackage{verbatim}

\usetikzlibrary[topaths, decorations]

\newcount\mycount
\pgfdeclaredecoration{simple line}{initial}{
  \state{initial}[width=\pgfdecoratedpathlength-1sp]{\pgfmoveto{\pgfpointorigin}}
  \state{final}{\pgflineto{\pgfpointorigin}}
}
\tikzset{
   shift left/.style={decorate,decoration={simple line,raise=#1}},
   shift right/.style={decorate,decoration={simple line,raise=-1*#1}},
}

\newcommand{\N}{\mathbb{N}}

\newcommand{\ind}{\mathbbm{1}}

\newcommand{\p}[1]{\left(#1\right)}
\newcommand{\abs}[1]{\left\vert#1\right\vert}
\newcommand{\floor}[1]{\left\lfloor #1 \right\rfloor}
\newcommand{\Gmat}{ \Gamma}

\newcommand{\Laplvec}[1]{  L^{(#1)}}

\newcommand{\LL}{\mathcal L}
\newcommand{\OO}{\mathcal O}
\newcommand{\TT}{\mathcal T}

\newcommand{\CC}{\mathcal C}

\newcommand{\SSS}{\mathcal S}

\newcommand{\dleq}{\preceq} 
\newcommand{\dl}{\prec} 
\newcommand{\dg}{\succ} 
\newcommand{\dgeq}{\succeq} 
\newcommand{\q}[1]{``#1''}
\newcommand{\Zth}{Z}
\newcommand{\Yth}{Y}

\newcommand{\Ythphi}{Y^\phi}

\newcommand\Wtilde{\stackrel{\sim}{\smash{W}\rule{0pt}{1.15ex}}}

\makeatletter
\def\namedlabel#1#2{\begingroup
    #2%
    \def\@currentlabel{#2}%
    \phantomsection\label{#1}\endgroup
}
\makeatother

\newtheoremstyle{mystyle}
{10pt} 
{10pt} 
{\normalfont} 
{} 
{\bfseries} 
{.} 
{.5em} 
{} 

\newtheoremstyle{break}
  {10pt}
  {10pt}
  {\normalfont}
  {}
  {\bfseries}
  {.}
  {\newline}
  {}

\theoremstyle{mystyle}
\newtheorem{remark}{Remark}
\newtheorem{claim}{Claim}
\newtheorem{theorem}{Theorem}
\newtheorem{lemma}{Lemma}
\newtheorem{proposition}{Proposition}
\newtheorem{condition}{Condition}
\newtheorem{defin}{Definition}

\newtheorem{assum}{Assumption}

\title{The real-time growth rate of stochastic epidemics on random intersection graphs}
\author{Carolina Fransson}

\begin{document}

\maketitle

\begin{abstract}

This paper is concerned with the growth rate of SIR (Susceptible-Infectious-Recovered) epidemics with general infectious period distribution on random intersection graphs.  This type of graph is characterized by the presence of cliques (fully connected subgraphs). 
We study epidemics on random intersection graphs with a mixed Poisson degree distribution and show that in the limit of large population sizes the number of infected individuals grows exponentially during the early phase of the epidemic, as is generally the case for epidemics on asymptotically unclustered networks. The Malthusian parameter is shown to satisfy a 
variant of  the classical Euler-Lotka equation. To obtain these results we
construct a coupling of the epidemic process and a continuous-time  multitype branching process, where the type of an individual is (essentially) given by the length of its infectious period. 
Asymptotic results are then obtained via an 
embedded single-type Crump-Mode-Jagers branching process.

\end{abstract}


\section{Introduction}

In the earliest epidemic models, it is  assumed  that the disease spreads in a population consisting of homogeneous individuals exhibiting homogeneous mixing. 
Since the advent of those early models, there has been considerable  interest in  incorporating realistic elements from real-world social structures that depart from the simplistic assumption of homogeneity. 
Such realistic features may take the form both of heterogeneity in social behaviour (some individuals may have a higher proclivity to be socially active than others, or the population may exhibit a more complex social structure than homogeneous mixing) and of biological differences in the \q{susceptibility} and \q{infectivity} of individuals. 

To give some examples, for  deterministic epidemic models this has been manifested through 
models where the population is stratified into a relatively small number of classes and individuals interact with each other at a rate that is determined by their classes. Individuals may, for instance, be spatially separated or stratified by age or sex. This typically  gives rise to a system of differential equations, which governs the dynamics of the epidemic \parencite{Watson72, mixture}.

A similar development of increasingly complex social structures has taken place in the field of stochastic epidemic modelling on networks. In particular,  a large body of epidemic models that aim to capture the tendency of individuals who know each other to have mutual acquaintances has appeared in the literature. In the context of models where the social network of the population is fully specified by a graph, this means that the graph is clustered (i.e. it contains a considerable amount of triangles and other short circuits).  Some examples include the great circle model \parencite{ball1997epidemics,greatcircle2003, greatcircleSIS2008} and the closely related small-world network model \parencite[]{smallworld1998}, where individuals typically have both local contacts in a local environment, which exhibits clustering, and global contacts.

In a similar vein, several models that include the presence of small closely connected  groups, or \textit{cliques}, with intense within-clique interactions has been introduced \parencite{Becker95, ball1997epidemics}. A clique may, for instance, represent a household, workplace or school.  
Models with this feature have been investigated in various forms, see for instance  \textcite{households2, households,overlapping,R0I,R0II, ball2012},
 to name a few.
 
 In the present paper, we study the real-time growth rate of an epidemic that spreads on a random graph whose structure, like that of the above-mentioned models, is characterised by the presence of 
 small, (possibly overlapping) highly connected cliques.  
 During the early phase of an epidemic, the number of infectious individuals typically grows exponentially; this is the case for many theoretical models and has also been observed in empirical data \parencite[]{NishiuraChowell2014, WHOteam2015}. 
 The  growth rate 
 is one of the most readily available attributes of an emerging epidemic and it is arguably one of the most natural parameters by which to describe the seriousness  
 of the epidemic.
 For many models of epidemics on random graphs with clustering, obtaining results that concerns the real-time-growth rate is however more challenging than analysing the final outcome of the  epidemic. The reason for this is that results on the final outcome of an epidemic may be obtained without taking the actual chain of transmission into account.  This idea was first mentioned in a paper devoted to epidemic modelling by \textcite{Ludwig} but was, however,  implicitly present in earlier literature on percolation \parencite{BroadbentHammersley1957, perc1}. For this reason, many models lend themselves more readily to analysis of the final outcome than of the real-time-growth rate.
 \textcite{PellisFergusonFraser2011} proposed approximate methods for estimating the so-called household reproduction number based on observations of the  real-time growth rate in a population structured into small (possibly overlapping) communities, both in the Markovian case and under the arguably strong assumption that 
 the total \q{infectivity} of an infectious individual and the time points at which the individual transmits the disease 
  are independent. 
On a related note, \textcite{BallShaw2015} provided methods to estimate the within-household infection rate for an SIR epidemic among a population of
households from the observed real-time growth rate.

As mentioned before, the real-time growth rate of an epidemic in a population with households, schools and workplaces has previously been studied in \parencite{PellisFergusonFraser2011}, where (among other things) heuristic results similar to those presented here were obtained.  
 In this paper, we provide rigorous proofs of these results. 
 It is worth to point out that the methods employed here can be applied to a more general class of random graphs with cliques than the model considered in this paper, and also a 
 more general class of household-school-workplace models than that studied in \parencite{PellisFergusonFraser2011}. In particular, previous results concern only the case where the clique or household sizes are bounded  and do not trivially extend to the setting with unbounded clique sizes, whereas the current paper deals with the unbounded case.

The key tool of this paper is a single-type branching process, which we embed in the epidemic process.  
Our approach is inspired by \textcite{MeinersIksanov}, where a similar embedding was used to obtain the polynomial rate of convergence of multi-type branching processes.  
The techniques employed here are also related to what \textcite{Sagitov2017} calls regenerative Galton-Watson processes and to the concepts of local infectious clumps and global contacts  in \textcite{overlapping}, see also  \textcite{Olofsson1996} which treats multitype branching processes with local dependencies.

To be more specific about the graph model, here we consider  the real-time growth rate of epidemics on a random intersection graph \parencite{intersection}. 
Simply put, a random intersection graph is constructed by dividing the nodes of the graph into groups (a node may belong to zero, one, or several groups) and then connecting  nodes that belong to the same group, 
so that the groups form fully connected (possibly overlapping) subgraphs. Thus, a random intersection graph does, in general, contain a non-negligible amount of short circuits, which makes the widely used branching process approximation  of the early phase of the epidemic somewhat delicate. 
Here we consider the real-time growth rate of epidemics on a random intersection graph  \parencite{intersection} in which the degrees distributions are mixed Poisson. 
 Epidemics on graphs of this type have previously been studied in \textcite{backward_adv}, where  expressions for the asymptotic probability of a major outbreak, the final size of a major outbreak and a threshold parameter were derived. Epidemics on random intersection graphs were also studied in \parencite{tunableclustering}, where the clustering of the underlying network is tunable. 



This paper is structured as follows. In section \ref{sec:notation} we present the notation conventions and abbreviations. Section \ref{sec:graphmod} contains an introduction to the underlying graph model and in section \ref{sec:epmod} we define the epidemic model. The main results are presented in section \ref{sec:main1}. Section \ref{sec:branching processtheory} and \ref{sec:proofs} contains some background theory and proofs of the main results.  


\section{Epidemics on random intersection graphs}

\subsection{Notation and abbreviations}\label{sec:notation}

This section contains a summary of notation conventions and abbreviations that will be frequently used in this paper. 

For any $B\subset \mathbb{R}$ and $x\in \mathbb{R}$ we use the notation $B_{\geq x}=B\cap [ x, \infty)$, and $B_{>x}$, $B_{\leq x}$ and $B_{< x}$ are defined analogously. 
For $x\in\mathbb{R}$, $\floor{x}=\sup\mathbb{Z}_{\leq x}$. 
For real numbers $x$ and $ y$, 
$x\vee y=\max(x,y) $ and 
$\log_+(x)=\log(1\vee x)$. 
For any $n\in \mathbb{Z}_ {\geq 1}$,  $[n]=\{1,\ldots, n\}$. 

Let 
$f:\mathbb{R}\to \mathbb{ R}$ and $g:\mathbb{R}\to \mathbb{ R}_{> 0}$. We write $f(x)=\OO(g(x))$ as $x\to \infty$ to indicate that $\limsup_{x\to\infty} \abs{f(x)}/g(x)<\infty$ and   $f(x)=o(g(x))$ as $x\to \infty$ to indicate that $\limsup_{x\to\infty} \abs{f(x)}/g(x)=0.$ Similarly,  $f(x)=\Uptheta(g(x))$ as $x\to \infty$ if $f(x)=\OO(g(x))$ and $\liminf_{x\to\infty} \abs{f(x)}/g(x)>0$.

For a random variable $X$ and an event $A$, we use the notation $E(X;A)=E(X\ind(A))$ where $\ind(A)$ is the indicator of $A$. 
We denote the mixed Poisson distribution with intensity $X$ by $M\!P(X)$, i.e. $Y\sim M\!P(X)$ means that $(Y\vert X=x)\sim Po(x)$. 
For any non-negative integrable random variable $X$ with  $E(X)>0$, we denote the size biased version of $X$ by $\bar X$, i.e.  for any Borel set $B\subset \mathbb{R}$
$$P(\bar X \in B)=\frac{E\left( X;X\in B\right)}{E(X)  }.$$

We will make frequent use of the abbreviations MP  (mixed Poisson) and SIR (Susceptible $\to$ Infectious $\to$ Recovered). 
Throughout this paper, $G_n$ denotes a random graph on $n$ vertices generated via the random intersection graph model. 
We say that an event occurs with high probability (w.h.p.) if the probability of the event tends to 1 as $n\to\infty$, where $n$ is the  number of vertices of the graph $G_n$ under consideration.

\subsection{The random intersection graph with mixed Poisson degrees}\label{sec:graphmod}
We consider a random intersection graph  model where the degrees of the nodes follow a mixed Poisson distribution. Epidemics on this particular type of graph have previously been investigated by 
\textcite{backward_adv}, who used a branching process coupling to derive expressions for the asymptotic probability of a major outbreak  (i.e. that a fraction $\Uptheta (1)$ of the population contracts the disease in the limit as $n\to\infty$, where $n$ is the population size), the final size of a major outbreak and a threshold parameter. In the present paper, we focus on  the (exponential) real-time growth rate of an epidemic on a random intersection graph in the early phase of a major outbreak. We will give a somewhat brief description of this graph model and refer the reader to \textcite{backward_adv} for a more detailed account. 

A graph $G_n$ on $n$ vertices can be constructed via the mixed Poisson  random intersection graph model as follows (see Figure \ref{fig:random intersection graph_constr} for an illustration of this construction). 
Let $A$ and $B$ be two random variables with expected values $E(A)=\mu_A$ and $E(B)=\mu_B$. We make the following assumption on $A$ and $B$.

\begin{assum}\label{assumMPrandom intersection graph}\begin{enumerate}
    \item[]
\end{enumerate}
 \begin{enumerate}[label=\roman*), ref=\roman*)]
 \item $P(A\geq 0)=P(B\geq 0 )=1$  \label{assumMPrandom intersection graphi}
 \item $P(A= 0)<1$ and $P(B=0 )<1$ 
 \item  $E(A^2\log_+A)<\infty$ and 
 $E(B^2\log_+B)<\infty$\label{assumMPrandom intersection graphiii}
 \end{enumerate}
\end{assum}

We will refer to the condition of assumption \ref{assumMPrandom intersection graph} \ref{assumMPrandom intersection graphiii} as the $x^2\log x$-condition. 

\begin{remark} 
This version of the random intersection graph  can be constructed under less strict assumptions on than the  $x^2\log x$-condition (see \textcite{backward_adv}), it is however needed here for the approximating branching process to satisfy the classical $x\log x$-condition.
\end{remark}

Let $\{A_k\}_k$ and $\{B_k\}_k$ be two sequences of independent copies of $A$ and $B$, respectively. Let further $V_{n}=\{v_1,\ldots, v_n\}$ be the vertex set of $G_n$, and assign the weight $A_i$ to the vertex $v_i$,  $i=1,\ldots,n$. As an intermediate  step, we construct an auxiliary graph $G_n^{\textrm{aux}}$ with vertex set $V_n\cup V_n'$, where  $V'_{n}=\{v_1', \ldots, v_m'\}$ 
and $$m=m(n):=\floor{n\mu_A/\mu_B} .$$ Assign the weight $B_j$ to the vertex $v_j'$, $j=1,\ldots, m$.  Given the weights of the vertices of $V_n$ and $V_n'$, for each pair $v_i, v_j'$ of vertices of $G_n^{\text{aux}}$ such that $v_i\in V_n$ and $ v_j'\in V_n'$ let the number of edges of $G_n^{\text{aux}}$ shared by $v_i$ and $v_j'$ have distribution
$$Po\left(\frac{A_iB_j}{n\mu_A}\right) ,$$
independently for  pairs $v_i,v_j'$. Thus  
in $G_n^{\text{aux}}$
the degree of $v_i\in V_n$ has distribution \begin{gather}\label{eq_po}
Po\left( A_i\frac{\mu_B^{(n)}\floor{ n\mu_A/\mu_B}}{n\mu_A}\right)
\end{gather}
where $\mu_B^{(n)}:={\sum_{j=1}^m B_j}/{m }$. Similarly, in $G_n^{\text{aux}}$
 the degree of $v_j'\in V_n'$ has distribution \begin{gather}\label{eq_poB}
Po\left( B_j\frac{\mu_A^{(n)}}{\mu_A}\right)
\end{gather}
where $\mu_A^{(n)}:={\sum_{i=1}^n A_i}/{n }$.
There are no edges of $G_n^{\text{aux}}$ between pairs  $v_{i_1},v_{i_2}\in V_n$. Similarly,  there are no edges of $G_n^{\text{aux}}$ between pairs of vertices of   $ V_n'$. 

We now obtain the graph $G_n$ from $G_n^{\text{aux}}$ by letting two distinct vertices $v_{i_1}, v_{i_2}\in V_n$ of $G_n$ share an edge if and only if $v_{i_1}$ and $ v_{i_2}$ of $G_n^{\text{aux}}$ have at least one common neighbour in $V_n'$.  
Next, we replace each edge of the undirected graph $G_n$ by two directed edges pointing in the opposite direction. The reason for this modification is that in the epidemic model considered in this paper (see Section \ref{sec:epmod}) infectious contacts are directed. 

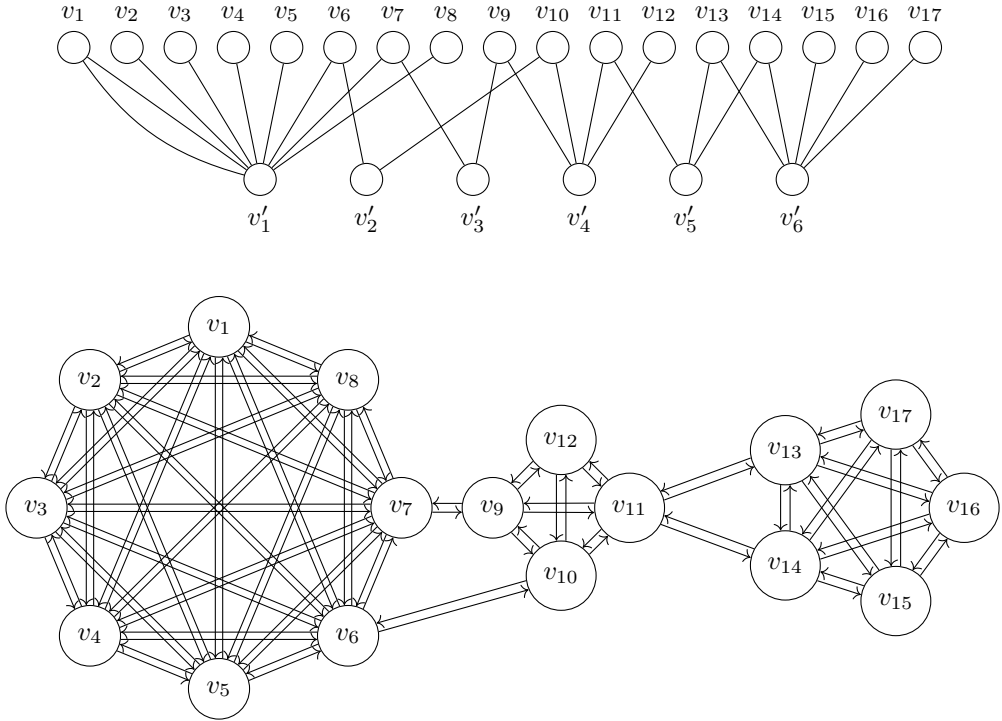
\begin{figure}[!ht]
\centering
\begin{tikzpicture}[scale=0.7]
\begin{scope}[shift={(-8.5,0)}]
  \foreach \number in {1,...,17}{
        \mycount=\number
      \node[draw,circle,inner sep=0.15cm, label={$v_{\number}$}] (N-\number) at (\mycount,0) {};
    }
\end{scope}    
   
\begin{scope}[shift={(-6,-2.5)}]
  \foreach \number in {1,...,6}{
        \mycount=\number
  		\multiply\mycount by 2
      \node[draw, circle, inner sep=0.15cm, label=below:{$v_{\number}'$}] (N'-\number) at (\mycount,0) {};
    }

\end{scope}

\foreach \number in {1,...,8}{
	\path (N-\number) edge[-,bend right=0] (N'-1);
    }

\foreach \number in {9,...,12}{
	\path (N-\number) edge[-,bend right=0] (N'-4);
    }

\foreach \number in {13,...,17}{
	\path (N-\number) edge[-,bend right=0] (N'-6);
    }

\path (N-6) edge[-,bend right=0] (N'-2);
\path (N-10) edge[-,bend right=0] (N'-2);

\path (N-7) edge[-,bend right=0] (N'-3);
\path (N-9) edge[-,bend right=0] (N'-3);

\path (N-11) edge[-,bend right=0] (N'-5);
\path (N-13) edge[-,bend right=0] (N'-5);
\path (N-14) edge[-,bend right=0] (N'-5);

\path (N-1) edge[-,bend right=20] (N'-1);

\end{tikzpicture}

\vspace{.7cm}
\begin{tikzpicture}[transform shape]
  \foreach \number in {1,...,8}{
        \mycount=\number
        \advance\mycount by 1
  \multiply\mycount by 45
        \advance\mycount by 0
      \node[draw,circle,inner sep=0.15cm] (N-\number) at (\the\mycount:2.4cm) {$v_\number$};
    }

  \foreach \number in {1,...,7}{
        \mycount=\number
        \advance\mycount by 1
  \foreach \numbera in {\the\mycount,...,8}{
    \path (N-\number) edge[->,bend right=4, shift right=0.2mm] (N-\numbera) edge[<-,bend left=4, shift left=0.2mm] (N-\numbera);
  }
}

\begin{scope}[shift={(4.5,0)}]
\foreach \number in {9,...,12}{
        \mycount=\number
        \advance\mycount by 1
  \multiply\mycount by 90
        \advance\mycount by 0
      \node[draw,circle,inner sep=0.15cm] (N-\number) at (\the\mycount:0.9cm) {$v_{\number}$};
    }

  \foreach \number in {9,...,11}{
        \mycount=\number
        \advance\mycount by 1
  \foreach \numbera in {\the\mycount,...,12}{
    \path (N-\number) edge[->,bend right=4, shift right=0.3mm] (N-\numbera) edge[<-,bend left=4, shift left=0.3mm] (N-\numbera);
  }
}

\end{scope}

\begin{scope}[shift={(8.5,0)}]
\foreach \number in {13,...,17}{
        \mycount=\number
        \advance\mycount by -1
  \multiply\mycount by 72
        \advance\mycount by 0
      \node[draw,circle,inner sep=0.15cm] (N-\number) at (\the\mycount:1.3cm) {$v_{\number}$};
    }

  \foreach \number in {13,...,16}{
        \mycount=\number
        \advance\mycount by 1
  \foreach \numbera in {\the\mycount,...,17}{
    \path (N-\number) edge[->,bend right=4, shift right=0.3mm] (N-\numbera) edge[<-,bend left=4, shift left=0.3mm] (N-\numbera);
  }
}

\end{scope}

\path (N-7) edge[->,bend right=4, shift right=0.3mm] (N-9) edge[<-,bend left=4, shift left=0.3mm] (N-9);
\path (N-6) edge[->,bend right=4, shift right=0.3mm] (N-10) edge[<-,bend left=4, shift left=0.3mm] (N-10);

\path (N-11) edge[->,bend right=4, shift right=0.3mm] (N-13) edge[<-,bend left=4, shift left=0.3mm] (N-13);
\path (N-11) edge[->,bend right=4, shift right=0.3mm] (N-14) edge[<-,bend left=4, shift left=0.3mm] (N-14);

\end{tikzpicture}
\caption{Construction of $G_n$ for $n=17$ with $\abs{V_n'}=\abs{\{v_1',\ldots,v_6'\}}=6$ cliques. Top: the auxiliary graph $G_n^{\text{aux}}$. Bottom: the resulting directed graph $G_n$.}\label{fig:random intersection graph_constr}
\end{figure}

For later reference, we point out that each clique  $\CC=(V_{\CC}, E_{\CC})$ of $G_n$ may be viewed as a (directed) subgraph  of $G_n$. Here the vertex set $V_{\CC}\subset V_n$ consists of the neighbours (in $G_n^{\text{aux}}$) of  $v'$ where $v'\in V_n'$ is the vertex that corresponds to $\CC$, and  $E_{\CC}$ is the  edge set of the simple directed complete graph $\CC$ (that is, for any pair $u,v\in V_\CC$ of distinct vertices there are precisely two edges $(u,v), (v,u)\in E_\CC$, and $E_\CC$ contains no self-loops).

\subsection{The epidemic model}\label{sec:epmod}

We consider a stochastic SIR epidemic on $G_n$.  
In the SIR model, individuals are classified as {susceptible} (S), {infectious} (I) or {recovered} (R) depending on their current health status. An individual who is classified as infectious can transmit the disease to other individuals in the population; if an infectious individual contacts a susceptible individual then transmission occurs and the susceptible individual immediately becomes infectious. An infectious individual will eventually recover from the disease after some period of time, which we refer to as the \emph{infectious period} of the individual in  question 
(in our model we allow for the infectious period of an individual to be $\infty$, which means that once the individual in  question has contracted the disease it will remain infectious forever). Recovered individuals are fully immune to the disease; once recovered, an individual plays no further role in the spread of the disease.  
For simplicity, we assume that the epidemic starts with one initial infectious case and that the rest of the population is initially fully susceptible to the disease. 
We assume that the disease spreads in a population of size $n$, where the underlying social network of the population is represented by $G_n$. In our model, a \q{close  contact} (a contact that results in transmission if a susceptible individual is contacted by an infectious individual) can only occur between individuals who are neighbours in $G_n$. 
Throughout the paper, we will use the terms individual and vertex interchangeably.

To be precise, the epidemic process can be defined as follows. 
Let $ I$ be a random variable with support in $\mathbb{R}_{\geq 0}\cup\{ \infty\}$. Each vertex $v_i$ of $G_n$ is equipped  with an infectious period $I_i$, where $\{I_i\}_i$ is a sequence of independent copies of $I$. 
Let $\{T_{ij}\}_{ij}$ be a sequence of identically distributed  random variables with support in $[0,\infty)$ such that $T_{i_1j}$ and $T_{i_2j}$ are independent if $i_1\neq i_2$. Similarly, we assume that $T_{i_1j}$ and $I_{i_2}$ are independent if $i_1\neq i_2$. Here $T_{ij}$ represents the time elapsed from the event that $v_i$ contracts the disease (which might or might not occur) to the event that $v_i$ contacts $v_j$. In many standard models, $T_{ij}$ are exponetially distributed. 
For each (directed) edge $(v_i, v_j)$, we equip $(v_i, v_j)$ with the transmission weight $$T_{ij}':=\begin{cases}
T_{ij} & \text{if } T_{ij}\leq I_i\\
\infty & \text{if } T_{ij}> I_i.
 \end{cases}$$
The transmission weight $T_{ij}' $ represents the time elapsed from the event that $v_i$ contracts the disease to the event that $v_i$ makes an infectious contact with $v_j$, which results in the transmission of the disease to $v_j$ if $v_j$ is still susceptible. 
We make the following assumption on the distribution of $T'_{ij}$.

\begin{assum}\label{assum:nonexpl}
$P(T'_{ij}=0)<1/(\mu_{\bar A}\mu_{\bar B})$
and the distribution of $T'_{ij}$  is non-lattice (i.e. $P(T_{ij}'\in\{\infty, 0,s,2s,\ldots\})<1$ for any $s>0$).
\end{assum}

The first part of Assumption \ref{assum:nonexpl}, $P(T'_{ij}=0)<1/(\mu_{\bar A}\mu_{\bar B})$, ensures that the approximating branching process does not explode (i.e. that the branching process population does not grow infinitely large in finite time). 

A path $\varsigma=(v_{i_1}, v_{i_2}, ...,v_{i_k}) $ is any finite sequence of vertices of $G_n$ such that $(v_{i_{r}} ,v_{i_{r+1}})$ is an edge of $G_n$, $r=1,\ldots k-1. $ We define the length $\ell(\varsigma)$ of a path   $\varsigma=(v_{i_1}, v_{i_2}, ...,v_{i_k}) $ as $$\ell(\varsigma)=\sum_{r=1}^{k-1} T_{i_ri_{r+1}}'. $$
Denote the collection of all paths from a vertex $u$ to a vertex $v$ by $\Sigma_{uv}$. 
The distance (transmission time) from $u$ to $v$ is 
given by 
$$d(u,v):=\min_{\varsigma}\ell(\varsigma)$$
where the minimum is taken over all paths $\varsigma\in\Sigma_{uv }$. We make the conventions $d(u,u)=0$ and $d(u,v)=\infty$ if $\Sigma_{uv}$ is empty for two distinct  vertices $u$ and $v$. 
\begin{remark}
Strictly speaking, $d$ is a quasi-distance rather than a distance since it is not symmetric. We do, however, abuse terminology for convenience. 
\end{remark}
The initial infected case $u_*$ is then selected according to some rule; a common choice  which we will adhere to here is to select the initial case  uniformly at random. We assume that the initial case $u_*$ contracts the disease at time 0. 
The time evolution of an outbreak can now be fully specified as follows.
An individual $v_i$, $i=1,\ldots, n$, has contracted the disease at time $t\geq 0$ if and only if $d(u_*,v_i)\leq t$, and $v_i$ has recovered from the disease at $t$ if and only if $d(u_*,v_i)+I_i\leq t$.

We will also need the \emph{within-clique distance}. Let $\CC$ be a clique of $G_n$.  For two vertices $u $ and $v$ 
let $\Sigma^{\CC}_{uv}$ be the collection of paths from $u$ to $v$ restricted to $\CC$. That is, a path $\varsigma$ from $u$ to $v$ belongs 
$\Sigma^{\CC}_{uv}$  whenever every edge of $\varsigma$ is also an edge of $E_{\CC}.$ Note that $\Sigma_{uv}^\CC$ is empty if $u$ and $v$ are not both members of $\CC$. 
The distance from $u$ to $v$ restricted to $\CC$ is 
given by 
$$d_{\CC}(u,v):=\min_{\varsigma}\ell(\varsigma)$$
where the minimum is taken over all paths $\varsigma\in\Sigma_{uv }^{\CC}$. As before $d_{\CC}(u,u)=0$ holds whenever $u\in V_{\CC}$, and $d_{\CC}(u,v)=\infty$ if $\Sigma_{uv}^{\CC}$ is empty. 


For any clique $\CC$, we refer to the first individuals of $\CC$ to contract the disease as the \emph{primary cases of $\CC$}. That is, given that the disease reaches $\CC$ (i.e. $\min_{w\in V_\CC}d(u_*,w)<\infty$),  a vertex $u\in V_\CC$ is a 
primary case of $\CC$ if $d(u_*,u)=\min_{w\in V_\CC}d(u_*,w)$. If $v\in  V_\CC$  contracts the disease but is not not a primary case of $\CC$, we say that $v$ is a \emph{secondary case of $\CC$},
regardless of whether $v$ is infected directly by a primary case or via some other path (which may or may not go via $u$). It is worth to point out that the primary case is almost surely unique if the transmission weight distribution has no atoms.

\section{Main results}\label{sec:main1}

In this section we present the main results and give a rough outline of the ideas behind the proofs. These proofs rely on asymptotic results on finite-type branching processes \parencite{Nerman1981} via a coupling of an epidemic process on $G_n$ and a single-type branching process. 
We point out that 
a salient feature of branching processes is that the lives of individuals that belong to different branches of the branching process tree are independent (conditioned on their types in the multitype case). 
In the epidemic process, however, the infectious individuals in a clique \q{compete} in transmitting the disease to the remaining susceptible individuals in the clique.
Therefore, naive attempts to couple a finite-type branching process with the epidemic process will in general 
give rise to non-local dependencies between the individuals of the branching process tree.
To deal with this, we will employ the branching process embedding presented below.

In the early phase of an outbreak, 
the epidemic process on $G_n$ can be coupled with a branching process $\Zth$ with type space 
$$\TT_\theta:=\TT\cup\{\theta\},$$
where $\TT$ is the support of the generic infectious period $I$ and the extra point $\theta$ is an atom  for the reproduction kernel of $Z$ in the sense of \textcite[Def. 4.3]{nummelin}.  The main idea of this section is to embed a single-type branching process $\Yth$ in $\Zth$ by letting the type-$\theta$ individuals of $\Zth$ be the individuals of $\Yth$. This allows us to employ the almost sure asymptotic results (see Section \ref{sec:branching processtheory} for an overview) that are available for single-type branching processes.

In $Z$, the type of an individual that was infected via a clique $\CC$ of size $\abs{V_{\CC}}=2$  is taken to be $\theta$, and the type of any other individual (except the individual that corresponds to the initial infectious case)
is taken to be the length of its infectious period.
The individuals of $Z$ are divided into generations by attributing 
all secondary cases in a clique to the primary case of the clique in question, even though   
it may well be that 
a secondary case does not get  infected directly by the primary case. 
In other words, if we follow the epidemic trail from $v$ back to the initial case $v_*$ then the generation of $v$ is the number of cliques that has to be traversed to reach $v_*$ (including the cliques of $v$ and $v_*$). 
If the extinction probability of the approximating branching process $\Zth$ is strictly smaller than 1 then $\Zth$ is said to be \emph{supercritical}, and we say that we are in the supercritical regime.

Let $\SSS$ be the set of all individuals of $\Zth$. 
We assume that there is one individual $a_0$ of generation 0 from which every other individual of $\SSS$ stems.
We will refer to the common ancestor $a_0\in \SSS$ as \emph{the first ancestor } of $Z$. 
The law of the life of  $a_0$ is usually different from the laws of the other branching process individuals since the initial case is assumed to be selected uniformly at random from the population, whereas 
individuals of subsequent generations represent infectious cases whose degree distribution is size biased. 
 For this reason, the initial case is a member of  $ M\!P(A)$ cliques while 
 the number of cliques that a non-initial case is a member of is distributed as $\bar D$ cliques, where $\ D\sim M\!P(A)$. 
 The following claim is easily checked (see also \textcite{backward_adv}). 

\begin{claim}
If $D\sim M\!P(A)$ then $(\bar D-1)\sim M\!P(\bar A).$

\end{claim}

Thus,  a non-initial case is a member of approximately  $  M\!P(\bar A)$ cliques that are not yet affected by the disease. 
By a similar reasoning the size of a clique (excluding the primary case of the clique) reached during the early phase of the epidemic has approximately distribution $M\!P(\bar B)$.

Now, let $F_k,\ k\in\mathbb{N}_{\geq 2}$, be the cumulative distribution function of the transmission time from the primary case in a clique of size $k$ to another (specific) member of the clique. That is to say, for a clique $\CC$ with $\abs{V_{\CC}}=k$ and two fixed individuals $u, v\in V_{\CC}$, $F_k$ is the cumulative distribution function of the within-clique distance $d_{\CC}(u,v)$. 
Let $p^{\bar B}_k:=P(M\!P(\bar B)=k)$ and define the vector 
 $\Gmat^T=\mu_{\bar A}(1\cdot p_1^{\bar B},2\cdot p_2^{\bar B},\ldots)$\label{gammavec} where $\mu_{\bar A}=E(\bar A)=E(A^2)/\mu_A$. For a clique $\CC$ of  size $\abs{V_{\CC}}=k$ and any $\lambda\geq 0$  let 
\begin{gather}\label{lapltransf}
\LL_k^{(\lambda)}:=\int_{\mathbb{R}_{\geq 0}}e^{-t\lambda}F_k(dt)
\end{gather}
  be the Laplace transform of the transmission time within $\CC$, 
  and  define the vector 
\begin{gather}\label{eq:laplacedef}
\Laplvec{\lambda}=(
\mathcal L_{2}^{(\lambda)},
\mathcal L_{3}^{(\lambda)},
\ldots)^T.     
\end{gather}

As we will prove in Section \ref{sec:proofs}, the Malthusian parameter can be found by solving the Euler-Lotka equation: 

\begin{defin}\label{def:Malt}
The \emph{Malthusian parameter} is the unique solution $\alpha>0$  of 
$ \Gmat\cdot \Laplvec{\alpha} =1$.
\end{defin}

We are now ready to state our main results.  

\begin{theorem}\label{main1}
Under Assumptions \ref{assumMPrandom intersection graph} \ref{assumMPrandom intersection graphi}-\ref{assumMPrandom intersection graphiii} and \ref{assum:nonexpl} and 
if the Malthusian parameter $\alpha>0$ exists 
then 
there exists a non-negative integrable  random variable $W$ such that
$$\frac{\abs{\Zth_t}}{e^{\alpha t}}\overset{a.s}{\to} W$$
 where $W$ satisfies  $P(\{\abs{\Zth_t}\not\to 0\}\Delta \{ W>0\} )=0$.
\end{theorem}

Here $\Delta$ denotes the symmetric difference, i.e. $A\Delta B=(A\setminus B)\cup (B\setminus A)$ and $\abs{ Z_t}$ denotes the number of individuals of $Z$ that are alive at time $t$. 
In the notation $\{\abs{\Zth_t}\not\to 0\}$ it is implicit that the limit is taken as $t$ tends to $\infty$, i.e. $\{\abs{\Zth_t}\not\to 0\}$ is the event that the branching process population of $Z$  ultimately avoids extinction.

\begin{theorem}\label{main2}

Let $(G_n)_n$ be a sequence of graphs generated via the random intersection graph model and assume that  the assumptions of Theorem \ref{main1} hold and let  $(\varepsilon_n)_{n\geq 1}$  be a sequence in $(0,\infty)$  that satisfies $\varepsilon_n\log(n)\to \infty$ as $n\to 0$. 
Then for any $q\geq 2$, $q\neq 3$, that satisfies 
$E(A^{q})<\infty $ and $E(B^{q})<\infty $  there exist couplings of the epidemic process on the $G_n$ and $\Zth$ such that  
 the two processes agree w.h.p.~until at least $n^{\gamma-\varepsilon_n}$ individuals have contracted the disease. Here $\gamma=\min\p{\frac{1}{2}, \frac{q-3/2}{q}}  $.

\end{theorem}

\section{Branching process theory} 
To present the main idea that underpins the proof of Theorem \ref{main1}, we need a  slightly more formal description of $Z$ (see e.g. \textcite{Jagers1989} for a full  formal framework).  
We assume that each individual of the branching process tree $Z$ has a countable infinite number of children, which makes $\SSS$ countable. 
The basic probability space on which $Z$ is defined is given by the product  probability space
$$\prod_{v\in\SSS}(\Omega_v, \mathcal A_v).$$
On each $(\Omega_v, \mathcal A_v)$ there is defined 
a point process $\xi_v$ on 
$[0,\infty]\times \TT_\theta$, which we refer to as the reproduction point process of $v$;  a point $(t, r)$ of $\xi_v$ corresponds to 
a type-$r$ individual produced by $v$ at age $t$, and there is a one-to-one correspondence between the children of $v$ and the points of $\xi_v$. Note that  $t=\infty$ is allowed; this has the interpretation that $v$ and its descendants are never born. 
For each $u\in \SSS $, let $\tau_u\in [0,\infty]$ be the  time point of $u$'s birth. 
We say that $u\in \SSS$ is \emph{realized} if $\tau_u<\infty$,  
and write $\abs{v}=n$ to indicate that an individual $v\in\mathcal S$ of $\Zth$ belongs to generation $n$, $n\geq 0$, and for any $v\in\SSS$ we denote the type of $v$ by $\sigma(v)\in\TT_\theta$.

The law of $\xi_v$ can be described as follows.  Given the type of an individual $v$, 
 its point process of reproduction is independent of the lives of the individuals that do not stem from $v$. 
A point $(t, r)$ of $\xi_v$ corresponds to a secondary type-$r$ case $v\in V_\CC$ for which the time elapsed since the corresponding primary case $u$ contracted the disease is $d_\CC(u,v)= t$. 
Each individual $v\in\SSS$ is assigned an infectious period $I_v$; given the type $\sigma(v)$ of $v$, $I_v$ is an independent copy of the generic infectious period if either $\sigma(v)=\theta$ or $v=a_0$ (i.e. $v$ is the first ancestor) , and $I_v=\sigma(v)$ otherwise. For $a_0$, the points of $\xi_{a_0}$ corresponds to the secondary cases where the corresponding primary case has infectious period $I_{a_0}$ and is a member of $  M\!P(A)$ cliques of independent sizes with distribution $M\!P(\bar B)$. Similarly, if  $v\neq a_0$ the points of $\xi_{v}$ corresponds to the secondary cases where the corresponding primary case has infectious period $I_{v}$ and is a member of $  M\!P(\bar A)$ cliques, each of size $M\!P(\bar B)$. The number $\abs{Z_t}$ of individuals of $Z$ that are alive at time $t\geq 0$,  correponds to the number of infectious individuals at $t$ and is given by 
is the cardinality of the set 
\begin{equation}\label{eq:Z_abs}
\{u\in\SSS: \tau_u\leq t < \tau_u+I_u \}.
\end{equation}

Before proceeding, we introduce some additional  notation. The individuals of $\Zth$ can be partially ordered by descent; we write $x\dleq y$ (or equivalently $y\dgeq x$) to indicate that $x$ is an ancestor of $y$ (we make the convention that an individual is an ancestor of itself) and $x\dl y$ (or $y\dg x$)  to indicate that $x\dleq y $ and $x\neq y$.  
Similarly, for $\mathcal J\subset \mathcal S$ and $x\in\mathcal S$ we write $ \mathcal J\dl x$ to indicate that $y\dl x$ for some $y\in\mathcal J$, and $ \mathcal J\dleq x$ to indicate that $y\dleq x$ for some $y\in\mathcal J$.

To arrive at Theorem \ref{main1}, we embed a single type branching process $\Yth$ into the above described branching process $\Zth$. 
To this end, we partition the individuals of $\Zth$ into \emph{blocks}. 
Let $\mathcal S_\theta\subset \mathcal S$ be the set of the  type-$\theta$ individuals of $\Zth$. 
For $x\in\SSS_\theta$, define the {block} $\mathscr B_x$ as follows: 
\begin{gather}\label{block}
\mathscr B_x:=\{y\in\SSS:x\dleq y, \text{ and whenever } x\dl z\text{ for some } z\in\SSS_\theta \text{ then }  z\not \dleq y \}.\end{gather}
In words, for any $x\in\SSS_\theta$ the block $\mathscr B_x$ is the set of descendants of $x$ for which the line of descent back to $x$ does not contain an individual of type $\theta$. 
The embedded branching process $\Yth$ is then obtained by letting the individuals of $\Yth$ be the individuals of $\mathcal S_\theta$ and the children of $x\in\SSS_\theta$ (seen as an individual of $\Yth$) the type-$\theta$ children of individuals of $\mathscr B_x$ (in $\Zth$).
That is, if we define $\mathscr J_n$, $n\geq 1$,  recursively as follows
$$\mathscr J_0=\{a_0 \}$$
and   
$$\mathscr J_n=\{ x\in \mathcal S_\theta: x\dg\mathscr J_{n-1} \text{ and whenever }   \mathscr J_{n-1}\dl z\dl x \text{ it holds that }  z\not \in \mathcal S_\theta\},$$
then $\mathscr J_n$ consists  of the individuals of generation $n$ of $\Yth$, $n\geq 0$. 

Since we are interested in the number of infected individuals at each time point $t\in\mathbb{R}_{\geq 0}$ we count the population of the embedded single type branching process $\Yth$ with a certain random characteristic (see e.g. \textcite{Nerman1981}) which provides the link between the size $\abs{\Zth_t}$ of the branching process population of $\Zth$ at $t$ and the embedded branching process $\Yth$. 
Here  we consider the special case where the random characteristic $\phi $ is defined as 
\begin{gather}\label{charac}
\phi_x(t)=\abs{\{y\in \mathscr B_x :\tau_y\leq t<\tau_y+I_y\}}
\end{gather}
for each $x\in \mathscr J=\cup_{n\geq 0}\mathscr J_n$
and we say that
$$\Ythphi_t:=\sum_{x\in\mathscr{J}}\phi_x(t-\tau_x)$$
is the branching process population of $Y$ {counted with the characteristic} $\phi$.
In words, $\phi_x(t)$ can be thought of  as the number of infectious individuals which belongs to the block $\mathscr B_x$ at $\tau_x+t$, where $\tau_x$ is the time point when $x$ contracts the disease. 
Thus, the total population size $\abs{ \Zth_t}$ of the approximating branching process $\Zth$ at the time point  $t$ can be recovered from the embedded single-type branching process $Y$  via the relation 
\begin{gather}\label{rel}
{\abs{\Zth_t}}={ \Ythphi_t}
\end{gather}

where $\Ythphi_t=\sum_{x\in\mathscr{J}}\phi_x(t-\tau_x)$
is the branching process population of $Y$ {counted with the characteristic} $\phi$ defined in (\ref{charac}) at $t$.

\begin{remark}
It is worth to point out that the embedding technique employed here does not require the presence of cliques of size 2 in the underlying graph model. Indeed, in a more general setting, an embedding of a single type branching process may be obtained by letting the individuals of the single type branching process be represented by the  vertices that are the last to be infected in their clique. 
This embedding relies on the observation that 
if a clique  $\CC$ (of size $\abs{V_\CC}=d\geq 2$ say) is fully infected then
the 
$d$th individual of $\CC$ to be infected does not compete with the other infected cases of $\CC$ in transmitting the disease to the remaining susceptible individuals of $\CC$. 
Thus, given that $v$ is the $d$th infected case of $\CC$, 
the infectious period of $v$ is independent of 
the actual paths of transmission within $\CC$. 
\end{remark}

\subsection{Branching processes counted with random characteristics}\label{sec:branching processtheory}
This section contains a brief overview of some preliminaries from the theory of branching processes, which we include for completeness.
More detailed accounts of the branching process theory can be found in \textcite{Nerman1981}
 and also in the more recent paper by \textcite{MeinersIksanov}. %
 
We begin by stating an asymptotic result for single-type branching processes where the  type of the ancestor is the same as the type of the other individuals. 
To this end, let  $\tilde Z$ be a branching process that behaves like a copy of $Z$,   
where $Z$ is the branching process in Section  \ref{sec:main1}, except that the first ancestor of $\tilde Z$ is of type $\theta$. Let further $\tilde Y$ be the corresponding embedded single-type branching process. In what remains, we will recycle the notation from section  \ref{sec:main1} for ease of notation. That is, we denote the type space of $\tilde Z $ by $\TT_\theta$, $\SSS$ denotes the space of individuals of $\tilde Z$, the block $\mathscr B_x$ and 
 the random characteristic $\phi$ are 
 analogous to the definitions in
(\ref{charac}) and (\ref{block}), respectively, 
and so forth.

 Let the random measure $\xi$ be the point process of reproduction on $\mathbb{R}_{\geq 0}$ of a generic individual of the single-type branching process $\tilde Y$, and let $ \xi^{(\alpha)}=\int_{\mathbb{R}_{\geq 0}}e^{-\alpha t} \xi(dt) =\sum_{x\in \mathscr J}e^{-\alpha\tau_x}$ where $\mathscr J=\cup_{n\geq 0} \mathscr J_n$ is the space of all individuals of $\tilde Y$ and $\alpha$ is the Malthusian parameter, i.e. $E(\xi^{(\alpha)})=1$. We define the measure $\nu$ on $\mathbb{R}_{\geq 0}$ by $\nu(t)=\nu[0,t]:=E(\xi(t)).$
Theorem \ref{thrm:asympt} stated below is a special case of  \textcite[Theorem 5.4]{Nerman1981} and  will lead us to the  a.s. convergence of Theorem \ref{main1}.  In order to  state Theorem \ref{thrm:asympt} we  need the following conditions.

\begin{condition}[Finite mean age at childbearing]\label{condbeta}
The mean age at childbearing $\beta$ defined by 
$$\beta:= E\p{ \int_{\mathbb{R}_{\geq 0}}te^{-\alpha t} \xi(dt)  }$$
is finite. 
\end{condition}

\begin{condition}[$x\log x$]\label{cond_xlogx} The random variable 

$$\xi^{(\alpha)}\log_+\p{\xi^{(\alpha)}}$$
has finite expectation.
\end{condition}

\begin{condition}
\label{cond1}

There exists some non-negative real-valued  non-increasing integrable function $g$ such that 
$$\int_{\mathbb{R}_{\geq 0}}\frac{1}{g(t)}e^{-\alpha t} \nu(dt)<\infty.$$
\end{condition}

\begin{condition}
\label{cond2}
There exists some non-negative real-valued  non-increasing integrable function $g$ such that the expectation of $$\sup_{t\geq 0}\frac{1}{g(t)\wedge 1}e^{-\alpha t} \phi(t)$$
is finite.

\end{condition}

\begin{theorem}[Theorem 5.4, \textcite{Nerman1981}]\label{thrm:asympt}

Under conditions \ref{condbeta} to \ref{cond2} above it holds almost surely  that $$\frac{\abs{\tilde Y_t^\phi}}{e^{\alpha t}}\to \widehat W m_\infty$$   as $t\to\infty$, where the random variable $\widehat W$ has mean $E(\widehat W)=1$ and $P(\{\widehat W=0\})=P( \abs{\Zth_t}\to 0 )$, and $m_\infty\in (0,\infty)$ is a constant that depends on $\phi$. 
\end{theorem}

\begin{remark}\label{remark}
Under the conditions of Theorem \ref{thrm:asympt}, applying Theorem \ref{thrm:asympt} to each of the children of the first ancestor of $Z$ gives that 

\begin{equation}\label{eq:remark}
\frac{\abs{\Zth_t}}{e^{ \alpha t}}=\frac{\abs{ Y_t^\phi}}{e^{\alpha t}}\to  W:=(\widehat W^{(1)}e^{-\alpha \tau_1}+\ldots+\widehat W^{(J)}e^{-\alpha \tau_J}) m_\infty
\end{equation}
almost surely as $t\to\infty$, where $J$ is the number of children of the first ancestor of $Z$ that are born in $[0,\infty)$, the time points  $\tau_1,\ldots, \tau_J$ are the birth times of those children  
and $\widehat W^{(1)},\ldots, \widehat W^{(J)}$ are $J$ copies of $\widehat W$ (which are not independent in general).
\end{remark}


\section{Proofs}\label{sec:proofs}

In Section \ref{seq:proof1}, we prove Theorem  \ref{main1} by showing that there is a coupling of the branching process $Z$ and a single-type branching process whose Malthusian parameter is given in Definition \ref{def:Malt}. In Section \ref{seq:proof2}, we prove Theorem \ref{main2}. The main step in the proof is to establish upper bounds on the total variation distance of the degree distribution in (\ref{eq_po}) and $Po(\bar A)$ and of the distribution in (\ref{eq_poB}) and $Po(\bar B)$.

\subsection{Proof of Theorem \ref{main1}}\label{seq:proof1}

Recall that the random measure $\xi$ (defined in section \ref{sec:branching processtheory}) is the point process of reproduction on $\mathbb{R}_{\geq 0}$ of a generic individual of $\tilde Y$ and that the measure $\nu$ on $\mathbb{R}_{\geq 0}$ is defined as $\nu(t)=\nu[0,t]:=E(\xi(t)).$
Also recall that $\Gamma=(\gamma_k)_k$ is the vector with elements of the form $\gamma_k= \mu_{\bar A}k p_k^{\bar B}$ where $p^{\bar B}_k=P(M\!P(\bar B)=k)$  for $k\in\mathbb{Z}_{\geq 1}$ and that $\Laplvec{\alpha}$ is the vector displayed in (\ref{eq:laplacedef}).

\begin{lemma}\label{claim0}
In the supercritical regime, the Malthusian parameter $\alpha>0$ of $Y$ exists if and only if $P(T_{ij}'=0)<1/(\mu_{\bar A}\mu_{\bar B})$ and is then  the unique solution of $\Gmat\cdot \Laplvec{\alpha}=1$.
\end{lemma}

\begin{proof}[Proof of Lemma \ref{claim0}]

Below $*$ denotes convolution, i.e. for two cumulative distribution functions $F$ and $G$ 
$$F*G(t)=\int_{-\infty}^tG(t-s)F(ds).$$
We have

\begin{equation}
  \label{nudef}
  \begin{gathered}
    \nu(t)=\gamma_1F_2(t)+\sum_{r}\sum_{(m_1,\ldots, m_r)} \gamma_{m_1} \gamma_{m_2}\cdots \gamma_{m_r} F_{m_1+1}*\ldots*F_{m_r+1}*F_{2}(t)\gamma_1
  \end{gathered}
\end{equation}

where the sums run over $\mathbb{Z}_{\geq 1}$ and $\mathbb{Z}_{\geq 1}^r$.
Taking the Laplace transform of the right hand side in (\ref{nudef}) and writing this in vector form gives that the Malthusian parameter $\alpha$ is the solution of
\begin{gather}\label{eq:malt1}
\int_{\mathbb{R}_{\geq 0}}e^{-\alpha t}\nu(dt)=\gamma_1\mathcal{L}_{2}(\alpha)\sum_{n=0}^\infty \left(\Gmat_{\geq 2}\cdot \Laplvec{\alpha}_{\geq 3} \right)^n=1\end{gather}

where  $\Gmat_{\geq 2}^T= (\gamma_2,\gamma_3,\ldots)$ and the elements of $\Laplvec{\alpha}_{\geq 3}=(
\mathcal L_{3}^{(\alpha)},
\mathcal L_{4}^{(\alpha)},
\ldots)^T$ are defined in (\ref{lapltransf}).
Since $\Gmat_{\geq 2}\cdot \Laplvec{\alpha}_{\geq 3}<1$ (shown below), (\ref{eq:malt1}) reduces to 

$$\frac{\gamma_1\mathcal{L}_{2}(\alpha)}{1- \Gmat_{\geq 2}\cdot \Laplvec{\alpha}_{\geq 3} }=1.$$

That is,  
 $\Gmat \cdot \Laplvec{\alpha}=1. 
$

It remains to show that $\Gmat_{\geq 2}\cdot \Laplvec{\alpha}_{\geq 3}<1$. By the $x^2\log x$ assumption, $$\Gmat\cdot \Laplvec{\lambda} \leq \Gmat \cdot (1,1,\ldots, 1)^T<\infty$$ for all $\lambda\geq 0$, and since the approximating branching process is supercritical $\Gmat\cdot L^{(0)}>1$. 
Since $\Gmat\cdot \Laplvec{\lambda}$ is continuous and strictly decreasing in $\lambda$ with
$$\Gmat\cdot \Laplvec{\lambda}\to P(T_{ij}'=0)\mu_{\bar A}\mu_{\bar B}<1$$ as $\lambda\to\infty$,
 the Malthusian parameter exists and is unique.  
\end{proof}

To proceed  we  need some additional notation and terminology. 
For two kernels  $K_1$ and $K_2$ (defined on the same measurable space $(E, \mathcal E)$), we define the convolution kernel $K_1 K_2$ as 
$$K_1K_2(r, A):=\int_EK_1(r,ds)K_2(s, A),\ A\in\mathcal E, r\in E.$$ 
For any $m\geq 1$,
$K_1^m:=K_1K_1^{m-1}$ 
and  $K_1^0:=I$ where $I$ is the identity kernel $I(r, A):=\ind(r\in A)$ for any $A \in \mathcal E$.
If $f$ is a ($\mathcal E$-measurable) function on $E$ then we define the function $K_1f$ as  $$K_1f(\cdot):=\int_Ef(s)K_1(\cdot,ds ),$$
and similarly for a measure $\eta$ on $(E, \mathcal E)$ we define $\eta K(\cdot )=\int \eta(ds)K(s, \cdot)$. For any $A\in\mathcal E$, let $I_A$ be the kernel $I_A(r, B)=I(r, A\cap B). $

Recall that we denote (a generic copy of) the point process of reproduction on $\mathcal T_\theta\times\mathbb{R}_{\geq 0}$ of a type-$r$ individual  of $\tilde Z$ by $\xi_r$, and let $\mu(r,A\times B)=E(\xi_r(A\times B))$ be the expected number of offspring of a type in $A\subset \TT_\theta$ produced by a type-$r$ individual (born at time 0) in  $B\subset \mathbb{R}_{\geq 0}$.
For $\lambda\in\mathbb{R}$, define the kernel
$K_{(\lambda)}(r, ds\times dt):=e^{-\lambda t}\mu(r,ds\times dt).$
Let further  
\begin{gather}\label{eq:Khat}
\widehat K (r, ds):=\int_{\mathbb{R}_{\geq 0}}K_{(\alpha)}(r, ds\times dt),
\end{gather} 
 and let $$G_\theta =\sum_{n=0}^\infty ( I_{\{\theta \}^c}\widehat K)^n$$
be the potential kernel of $I_{\{\theta \}^c}\widehat K$. Here $\{\theta\}^c=\mathcal T_\theta\setminus\{\theta\}$ and   \begin{equation}\label{eq:I_op}
    I_{\{\theta\}^c}(r, B)=I(r, B\cap \mathcal T_\theta\setminus \{\theta\}).
\end{equation}
\begin{remark}
Note that for any $A\subset \TT$ and $s\in\TT_\theta$, 
$G_\theta(s, A)$ is the expected number of descendants of an individual $u$ of $\tilde Z$, $\sigma(u)=s$, that are members of the same block as $u$ and whose type belongs to $A$ 
discounted by their time of birth. Similarly, $G_\theta(s, \{\theta\})$ is the expected number of type $\theta$-individuals stemming from $u$ whose mother belongs to the same block as $u$ 
discounted by their time of birth.
\end{remark}

Define  the function $h$ by 
\begin{gather}\label{eq:h}
h(x)=G_\theta(x, \{\theta \})\end{gather}
and the measure $\pi$ by 
\begin{gather}\label{eq:pi}
\pi(A) = \widehat K G_\theta (\theta,A).
\end{gather}
Then  $h$ is harmonic for $\widehat K$ (see \textcite[Proposition 4.6]{nummelin}), i.e. $ \widehat K h=h$.  
Similarly, $\pi$ is invariant for $ \widehat K$, i.e. 
$\pi=\pi \widehat K$.

In order to use the finite-type-branching-process toolbox, we need to verify that the mean age at childbearing $\beta$ of 
$Y$ is finite, i.e. that
$\beta=\int te^{-\alpha t} \nu(dt)<\infty$ where $\nu$ is the measure in (\ref{nudef}). 
\begin{lemma}\label{claim00}
$0<\beta<\infty$.
\end{lemma}

\begin{proof}[Proof of Lemma \ref{claim00}]

Let $\varepsilon>0$ be small so that $ \Gmat_{\geq 2}\cdot \Laplvec{\alpha-\varepsilon}_{\geq 3} <1$, and let the constant $C_\varepsilon$ be such that $C_\varepsilon e^{ -(\alpha-\varepsilon)t} \geq t e^{ -\alpha t}$ for all $t\geq 0. $
Then 
$$\beta =\int t e^{ -\alpha t} \nu(dt) \leq  C_\varepsilon \int e^{ -(\alpha-\varepsilon)t} \nu(dt) $$
$$ = C_\varepsilon \gamma_1\mathcal{L}_{2}(\alpha-\varepsilon)\sum_{n=0}^\infty \left(\Gmat_{\geq 2}\cdot \Laplvec{\alpha-\varepsilon}_{\geq 3} \right)^n<\infty .$$
\end{proof}

\subsubsection[Optimal lines and the x log x condition]{Optimal lines and the $x\log x$ condition}

In this paper, we consider two ways of dividing the individuals of the approximating branching process $\tilde Z$
into generations; either generation $n$ consists of the individuals of  $\mathcal S_n:=\{x\in\mathcal S:\abs{x}=n\}$ (i.e. of the individuals separated from the first ancestor by a line of descent of length $n$), or generation $n$ consists of the individuals of $\mathscr J_n$ which leads us to the embedded branching process $\tilde Y$. There is a close connection between these two ways of viewing generations and the concept of \emph{ stopping lines} (see \textcite{Jagers1989} or \textcite{biggins2004}).

Following \textcite{Jagers1989}, we say that a set of individuals $L\subset \mathcal S$ is a \emph{stopping line } if for any pair $y,x\in L$ it holds $x\not \dl y$. In other words, a stopping line $L$ cuts across the branching process tree $Z$ in the sense that if
$x\in L$ then no descendants or ancestors of $x$ (apart from the individual $x$ itself) are members of $L$. 
For any $x\in \mathcal S$, let $\mathcal G_x$ be the $\sigma$-algebra generated by the lives (infectious periods and reproduction processes) of the ancestors of $x$ (including  $x$), 
 and for any non-random stopping line $\ell$, let $\mathcal G_\ell:=\sigma(\cup_{x\not\dgeq\ell} \mathcal G_x)$ be the $\sigma$-algebra generated by the lives of the individuals 
that do not have an ancestor in $\ell$. 
Mirroring the concept of optimal stopping times, we say that a line $L$ is \emph{optimal} if for any non-random stopping  line $\ell $ the event $\{ L\dleq \ell\}$ belongs to $\mathcal G_\ell$.
Here $ L\dleq \ell$ means that for any $x\in L$ there is $y\in \ell$ such that $x\dleq y $.

 Note that  $\mathscr J_n$
is an optimal line for each $n\geq 0$. Note also that $\mathcal S_n$ is an optimal line, $n\geq 0$.

For each $n\in\mathbb{Z}_{\geq 0}$ define
\begin{gather}\label{What}
 \widehat W_n:=\sum_{x\in\mathscr J_n} e^{-\alpha \tau_x}. 
\end{gather}
and  
\begin{gather}\label{W}
\Wtilde_n=\frac{1}{h(\theta)}\sum_{\abs u=n}e^{-\alpha \tau_u}h(\sigma_u)
\end{gather}
where the sums run over all individuals of generation $n$ of $\tilde Y$ and $\tilde Z$, respectively. 
It is well known \parencite{biggins2004} that $\{\Wtilde_n \}_{n\in \mathbb{Z}_{\geq 0} }$ is a martingale with respect to $\mathcal F:=\{\mathcal F_n \}_n$, where 
$\mathcal F_n:=\mathcal G_{\mathcal S_n}$ is generated by the lives of the individuals up to generation $n$ (of $\tilde Z$) for $n\geq 1$ (we make the convention that $\mathcal G_{\varnothing}$ is the trivial $\sigma$-algebra).  
Similarly, $\{ \widehat W_n \}_{n\in \mathbb{Z}_{\geq 0} }$ is a martingale with respect to $\{\mathcal G_{\mathscr J_n} \}_n$. 
 For later reference, we now state a special case of results presented in  \textcite{ biggins2004}.

\begin{proposition}[c.f. \textcite{ biggins2004}, Theorem 6.1 and Lemma 6.2]\label{Wlim}
Let $\{\widehat W_n\}_n$ and $\{ \Wtilde_n\}_n$ be as in (\ref{What}) and (\ref{W}). Then, with probability 1, the limits $\lim_n \Wtilde_n$ and $\lim_n \widehat  W_n$ exist 
 and
$$\lim_n \Wtilde_n=\lim_n \widehat  W_n. $$

\end{proposition}

Now, the  $x\log x$ condition for the single-type branching process $\tilde Y$ takes the form 
\begin{gather}\label{xlogx_ineq}
E( \widehat W_1\log_+ \widehat W_1)<\infty. 
\end{gather}

The following two lemmas assert that $\tilde Y$ satisfies the $x\log x$ condition under the $x^2\log x$-condition. 
\begin{lemma}\label{claim1}
 Let $J\sim M\!P(X)$, where $X$ is a non-negative integrable random variable with $P(X=0)<1$. Then $E(X^2\log_+X)<\infty$ is necessary and sufficient for $E(\bar J\log_+\bar J)<\infty$ to hold. 
\end{lemma}

\begin{proof}[Proof of Lemma \ref{claim1}]
First note that 

$$E(\bar J\log_+\bar J)=E( J^2\log_+ J)/E(J)=E( J^2\log_+ J)/E(X).$$

Since $x\mapsto x^2\log_+ x$  is convex, necessity now   follows from Jensen's inequality for conditional expectations. 
Sufficiency follows from 
$$E( J^2\log_+ J)=E\left( \sum_{k\geq 0} k^2\log_+k\frac{X^ke^{-X}}{k!}\right)$$
$$=E\left( X\sum_{k\geq 0} (k+1)\log_+(k+1)\frac{X^ke^{-X}}{k!}\right)$$
$$=E\left( X\sum_{k\geq 0}\log_+(k+1)\frac{X^ke^{-X}}{k!}\right)+E\left( X\sum_{k\geq 0} k\log_+(k+1)\frac{X^ke^{-X}}{k!}\right)$$
$$=E\left( X\sum_{k\geq 0}\log_+(k+1)\frac{X^ke^{-X}}{k!}\right)+E\left( X^2\sum_{k\geq 0} \log_+(k+2)\frac{X^ke^{-X}}{k!}\right)$$
$$\leq E\left( X\log_+(X+1)\right)+E\left( X^2\log_+(X+2)\right),$$

where the inequality follows from Jensen's inequality applied to the concave (on $\mathbb{R}_{\geq 0}$) functions $x\mapsto \log_+ (x+1)$ and $x\mapsto \log_+ (x+2)$. 
\end{proof}

\begin{lemma}\label{Wxlogx}

Under assumption \ref{assumMPrandom intersection graph} $\{ \widehat W_n\}_n$, satisfies the $x\log x$ condition  in (\ref{xlogx_ineq}). 
\end{lemma}

\begin{proof}
It is  known \parencite[Proposition 4.1]{MeinersIksanov} that the inequality in (\ref{xlogx_ineq}) holds if and only if $\{\widehat  W_n\}_n$ is uniformly integrable, which is also equivalent to $E(\widehat  W)=1$, where $\widehat  W$ is the almost sure limit $\lim_n  \widehat  W_n$.  
By Proposition \ref{Wlim}, $$\lim_n \widehat W_n=\lim_n \Wtilde_n$$
almost surely. Thus, in order to verify that $\{\widehat W_n\}_n$ is uniformly integrable it is sufficient to verify that the almost sure limit of $\Wtilde_n$ has mean 1. 

Now note that $h$ is bounded; for any $x\in \TT_\theta$ we have $$h(x) = G_\theta(x,\{\theta\})\leq 1+ E(\bar A)E(\bar B) \pi(\{\theta \})=1+ E(\bar A)E(\bar B).$$

Combining this with Assumption \ref{assumMPrandom intersection graph}, Lemma \ref{claim1} and  \textcite[Corollary 2.1]{MeinersIksanov}, the $x\log x$ condition holds for $Y$ if we can show that almost surely 
\begin{gather}\label{xlogxineq}
\sup_{x>2}\left(\frac{\sum_i \mathbbm{1} (H(\varsigma_i)\geq x^{-1})}{\log_+(x)}\right)<\infty
\end{gather}

where $\varsigma=(\varsigma_0,\varsigma_1, \ldots)$, $\varsigma_0=(\theta,0)$, is a markov chain on $\mathcal T_\theta\times \mathbb{R}_{\geq 0}$ with transmission 
measure given by

$$R((r ,0), A\times B):=\frac{1}{h(r)}\int_{A\times B} h(s) K(r, ds\times dt)$$

$$R((r ,t), A\times B ):=R((r ,0), A\times (B-t )_{\geq 0})$$

where $B-t=\{b-t:b\in B\}$ 
and 
$H((r, t))=e^{-\alpha t} h(r)$. 

Let $p_1:\mathcal T_\theta\times\mathbb{R}_{\geq 0}\to \mathcal T_\theta$ and  $p_2:\mathcal T_\theta\times\mathbb{R}_{\geq 0}\to \mathbb{R}_{\geq 0}$ be the projections onto the first and second coordinate, respectively, and put $\varsigma_j'=p_1(\varsigma_j)$ and $\varsigma_j''=p_2(\varsigma_j)$ for $j\geq 0$. 
Then $\{\varsigma_0', \varsigma_1',\ldots\}$ is a Markov chain on $\mathcal T_\theta$ with  transition measure $R_1$:
$$R_1(r, B)=\frac{1}{h(r)}\int_B h(s)K(r, ds\times \mathbb{R}_{\geq 0})= \frac{1}{h(r)}\int_B h(s)\widehat K(r, ds)$$
for (measurable) $B\subset \mathcal T_\theta$. 

Now, it is easily verified that $\theta$ is a (positive) recurrent state of $\{\varsigma_0', \varsigma_1',\ldots\}$. Indeed, let $M_i$, $i\geq 1$ be the number of steps until $\{p_1(\varsigma_0), p_1(\varsigma_1), \ldots\}$ revisits $\theta$ for the $i$th time, given that $p_1(\varsigma_0)=\theta$:  
$$M_0:=1$$ 
and for $i\geq 1$
$$M_i=\inf\{m > M_{i-1}:p_1(\varsigma_m)=\theta \}$$  
Then 
for $m\geq 1$ (recall that $I_{\{\theta\}^c}$ is the operator in  \ref{eq:I_op})

$$P(M=m)=R_1(I_{\{\theta^c\}}R_1)^{m-1}(\theta,\{\theta\} )$$
$$=\int_{\mathbb{R}_{\geq 0}^{m-1}}K(r_{m-1},\{\theta\}\times\mathbb{R}_{\geq 0})K(r_{m-2}, dr_{m-1}\times\mathbb{R}_{\geq 0})\ldots K(r_{1}, dr_{2}\times\mathbb{R}_{\geq 0})K(\theta, dr_{1}\times\mathbb{R}_{\geq 0})$$
$$=\gamma_1\mathcal{L}_{2}(\alpha) \left(\Gmat_{\geq 2} \cdot \Laplvec{\alpha}_{\geq 3} \right)^{m-1}=(1-\Gmat_{\geq 2} \cdot \Laplvec{\alpha}_{\geq 3}) \left(\Gmat_{\geq 2} \cdot \Laplvec{\alpha}_{\geq 3} \right)^{m-1}.$$

Thus $P(M=\infty)=0$ and $M$ is geometrically distributed.
Thus the inequality in (\ref{xlogxineq}) holds almost surely by the law of large numbers  applied to  $\{\varsigma_{M_k}''+\ldots \varsigma_{M_{k+1}-1}''\}_k$.
\end{proof}

We now turn our attention to conditions \ref{cond1} and \ref{cond2}. 

\begin{lemma}\label{lemmacond34}

Conditions \ref{cond1} and \ref{cond2} are satisfied under assumption \ref{assumMPrandom intersection graph}. 

\end{lemma}

\begin{proof}
If we take $g$ to be the mapping $t\mapsto e^{-\varphi x}$ then  conditions \ref{cond1} and \ref{cond2} are satisfied if $\varphi>0$ is small enough. 
Indeed, let $\varphi\in (0,\alpha)$ be small so that $\Gmat_{\geq 2}\cdot L^{(\alpha-\varphi)}_{\geq 3} <1$ and put $g(t)=e^{-\varphi t}$.  Then
$$\int_0^\infty \frac{1}{g(t)} e^{-\alpha t }\nu(dt)=\int_0^\infty e^{-(\alpha-\varphi) t }\nu(dt)=\gamma_1\mathcal{L}_{2}^{(\alpha-\varphi)}\sum_{n=0}^\infty \left(\Gmat_{\geq 2}\cdot \Laplvec{\alpha-\varphi}_{\geq 3}  \right)^n<\infty.$$

Below, $x$ is a generic type-$\theta$ individual of $\Zth$ with $\phi=\phi_x$, $\mathscr B=\mathscr B_x$ and $\tau_x=0$. 
Clearly, since $e^{-\alpha t}/g(t)$ is non-increasing in $t$ we have 
$$\frac{e^{-\alpha t }\phi(t)}{g(t)}\leq \sum_{y\in\mathscr B} \frac{e^{-\alpha \tau_y }}{g(\tau_y)}$$

for any $t\geq 0$. 
Thus
$$E\p{\sup_{t\geq 0}\frac{e^{-\alpha t }\phi(t)}{g(t)}}\leq E\p{\sum_{y\in\mathscr B} e^{-(\alpha-\varphi) \tau_y }}=\sum_{n=1}^\infty \left(\Gmat_{\geq 2}\cdot L^{(\alpha-\varphi)}_{\geq 3}\right)^n<\infty.$$
If the random characteristic $\phi$ is as in (\ref{charac}) then condition \ref{cond2} is satisfied for the same choice of $g$.

\end{proof}

\begin{proof}[Proof of Theorem \ref{main1}]

Assume that Assumption \ref{assumMPrandom intersection graph} \ref{assumMPrandom intersection graphi} - \ref{assumMPrandom intersection graphiii} hold.  By  Lemma \ref{claim0}, the Malthusian parameter for the single type branching process $\tilde Y$ is  the unique solution $\alpha>0$  of 
$ \Gmat\cdot \Laplvec{\alpha} =1$.
By Lemma \ref{claim00} Condition \ref{condbeta} holds, by Lemma \ref{Wxlogx} Condition \ref{cond_xlogx} holds and  by Lemma \ref{lemmacond34} Conditions \ref{cond1} and \ref{cond2} hold.  
Hence the conditions of Theorem \ref{thrm:asympt} are satisfied,  and by Remark \ref{remark} the convergence in (\ref{eq:remark}) holds.

\end{proof}

\subsection{Proof of Theorem \ref{main2}}\label{seq:proof2}

Below we describe a probabilistically equivalent construction of $G_n^{\text{aux}}$. This alternative way of constructing $G_n^{\text{aux}}$ is  useful in the branching process approximation of the epidemic process since it allows us to run the epidemic process and construct  $G_n$ 
in unison. 

Throughout this section we denote the probability measure of $A$ by $p$, that is $p([0, x])=P(A\leq x)$ for  $x\in[0,\infty)$. Similarly, we denote the probability measure of the size biased version $\bar A$ of $A$ by $\bar p$. 
Given the weights $A_1,\ldots, A_n$, let $A_{(n)}$ be a random variable which follows the empirical distribution of $A_1,\ldots, A_n$ and let $p_n$ be the corresponding probability measure, i.e. 
\begin{gather}\label{eq:sibi}
p_n([0,x])=P(A_{(n)}\leq x\vert A_1,\ldots, A_n)=\frac{1}{n}\sum_{i=1}^n\ind(A_i\leq x)
\end{gather}
for $x\in[0,\infty)$.
 Let  further $\bar A_{(n)}$ denote the size biased version of $A_{(n)}$ and let $\bar p_n$ be the corresponding probability measure, i.e. 

$$\bar p_n([0, x])=\frac{1}{n\mu_A^{(n)}}\sum_{i=1}^nA_i\ind(A_i\leq x).$$

Similarly, conditioned on the weights $B_1,\ldots, B_m$, let $B_{(n)}$ be a random variable which follows the empirical distribution of $B_1,\ldots, B_m$ and let $\bar B_{(n)}$ be the size-biased version of $B_{(n)}.$

To construct $G_n^{\text{aux}}$, start by picking some vertex $u$ of $V_n$ according to some rule (e.g. uniformly at random).  Typically, $u$ represents the initial case of the epidemic. 
Put 
$\mathcal E_0=\mathcal E_0'=\mathcal N_0=\varnothing$.   
The component of  $G_n^{\text{aux}}$ that contains $u$ is now constructed by iterating the following steps for $t=1,2,\ldots$.  

\begin{enumerate} \label{constralg}
\item 
If $t=1$ let $v=u$  be the vertex that is currently being explored.  Generate the downshifted \emph{group degree} $D$ of $v$ from 
the distribution given in (\ref{eq_po}). Here $D$ represents the (additional) number of cliques that $v$ is a member of. 
\item 
Draw $D$ elements $ B_{(1)},\ldots,  B_{(D)}$ from the multiset $\{B_1,\ldots, B_m\}$ independently with replacement. 
The probability to select $B_k\in \{B_1,\ldots, B_m\}$ in a specific draw is given by $B_k/m\mu_B^{(n)}$, where $\mu_B^{(n)}=\sum_{k=1}^mB_k/m.$  In other words, we generate $D$ independent copies of $\bar B_{(n)}$.

\item \label{point}
Let $v'_{(j)}\in V_n'$ be the vertex that corresponds to $B_{(j)}$, $j=1,\ldots, D$. 
For each  $v_{(j)}'\in\{v_{(1)},\ldots, v_{(D)}\}$, if $v'_{(j)}$ is not a member of the set $  \mathcal E'_{t-1}\subset V_n'$ of hitherto explored vertices  carry out step \ref{first} to \ref{last} below.  If $v_{(j)}'\in \mathcal E'_{t-1}$ then the clique that corresponds to  $v_{(j)}'$ is already explored,  so $v_{(j)}'$ is excluded  from
the steps below. 
\begin{enumerate}[label=(\alph*),ref=\ref{point} (\alph*)]
\item Generate the downshifted clique size $D_j'$ of $v_{(j)}'$  from 
the distribution given in (\ref{eq_poB}) with $B_{(j)}$ in place of $B_j$. \label{first}
\item Select $D_j'$ elements $ A_{(1)}^j,\ldots,  A_{(D_j')}^j$ from the multiset $\{A_1,\ldots, A_n\}$ independently with replacement. The probability to select  $A_k$ in a specific draw is given by $A_k/n\mu_A^{(n)}$, where $\mu_A^{(n)}=\sum_{k=1}^nA_k/n.$ In other words, we generate $D'_j$ independent copies of $\bar A_{(n)}$. 
\label{prev}
\item 
Let $v_k^j\in V_n$ be the vertex that corresponds to $A_{(k)}^j$, $k=1,\ldots, D_j'$. 

\item Add  an edge between each pair of distinct vertices in  $\{v\}\cup\{v_k^j\}_{k=1}^{D'_j}\setminus \mathcal E_{t-1}$ to $G_n$. 
\label{last}
\end{enumerate}
\item 
Update the set $\mathcal E_t'$ of  explored cliques by putting  $\mathcal E_{t}'=\mathcal E_{t-1}'\cup\{v'_{(1)},\ldots,  v'_{(D_i)}\}$.

\item 
Update the set $\mathcal N_t$ of neighbours of explored vertices by putting  $\mathcal N_{t}=\mathcal N_{t-1}\cup \{v_k^j:j=1,\ldots,D_i \text{ and } k=1\ldots, D'_j \}$.
\item 
Update the set $\mathcal E _t$ of  explored vertices by putting  $\mathcal E_{t}=\mathcal E_{t-1}\cup \{v \}$.

\item If $\mathcal N_t=\mathcal E_t $ then the construction of the component is complete and we exit the algorithm. Otherwise, update $v$ by picking some new vertex in $ \mathcal N_t\setminus\mathcal E_t $. If   $G_n^{\text{aux}}$ is constructed as the epidemic progresses, then  the  new vertex $v$ is the $t$th non-initial case). 

\end{enumerate}

If the nodes   $v_{(1)}',\ldots, v_{(D)}'$  are not distinct  or $\{v_{(1)}',\ldots, v_{(D)}'\} \cap \mathcal E_{t-1}\neq \varnothing$  in step \ref{point}  for some iteration $t\in\{0,1,\ldots \}$ then  the coupling of the approximating branching process and the epidemic process breaks down. 
Similarly, if in some iteration $t$ the  $v_k^j$ are not all distinct or $v_k^j\in E_{t-1}$ for some $j\in\{1,\ldots, D\}$ and $k\in\{ 1,\ldots, D_j'\}$ then the coupling breaks down.

The following claim, which is a variant of the birthday problem,  ensures that with high probability the coupling of the approximating branching process and the epidemic process holds during the first $o(\sqrt n)$ steps of the construction algorithm. 
The proof is included here for completeness.  

\begin{claim}\label{claim:coupling}
 Let  $j_n=o(\sqrt{n})$ and let $A$ have finite second moment. Suppose that  we draw elements   from the multiset $\{A_1,\ldots,A_n\}$ independently with replacement, and that the probability  that  $A_i\in\{A_1,\ldots,A_n\}$ is selected in a specific draw is proportional to $A_i$. It then holds  that 
the first $j_n$ drawn  elements are distinct  with $P$-probability tending to 1 as $n\to\infty$. 
 
\end{claim}

\begin{proof}[Proof of Claim \ref{claim:coupling}]

For $k=1,2,\ldots$, let $A^{(k)}$ be the $k$th element that is drawn from $\{A_1,\ldots, A_n\}$ and let $E_n(\cdot)$ be the conditional  expectation operator given $A_1,\ldots, A_n$. 
Conditioned on $A_1,\ldots, A_n$, for $k\geq 2$ the probability that $A^{(k)}$ is not distinct from $A^{(j)}$ for some $j\in\{1,\ldots, k-1\}$ is smaller  than or equal to 

$$E_n\p{\frac{A^{(1)}+\ldots+A^{(k-1)}}{A_1+\ldots+A_n} }=(k-1){\frac{A_1^2+\ldots +A_n^2}{\p{A_1+\ldots+A_n}^2} }.$$

Thus, by the union bound, conditioned on $\{A_1,\ldots, A_n\}$ the probability that  the first $j_n$ drawn elements are not distinct is  smaller than or equal to

$$\sum_{k=1}^{j_n}(k-1){\frac{A_1^2+\ldots +A_n^2}{\p{A_1+\ldots+A_n}^2} } $$

$$=\frac{(j_n-1)j_n\p{A_1^2+\ldots +A_n^2}}{2\p{A_1+\ldots+A_n}^2}$$

$$=\p{\frac{(j_n-1)j_n}{n}}\p{\frac{\p{A_1^2+\ldots +A_n^2}/n}{2(\mu_A^{(n)})^2}}$$

Since $A$ has finite second moments 

$$\frac{\p{A_1^2+\ldots +A_n^2}/n}{2(\mu_A^{(n)})^2}$$

converges to $E(A^2)/2\mu_A^2$ in $P$-probability as $n\to\infty$. Hence 

$$P(A^{(1)},\ldots, A^{(j_n)}\text{ are distinct})$$

$$\geq 1-E\p{1\wedge\p{\frac{(j_n-1)j_n}{n}}\frac{\p{A_1^2+\ldots +A_n^2}/n}{2(\mu_A^{(n)})^2} }$$

where the right side tends to 1 as $n\to\infty$. 
\end{proof}

In the  construction algorithm  described above,  the weights of explored vertices in $V_n$ and $V_n'$  are  generated from the distributions of $\bar A_{(n)}$ and $\bar B_{(n)}$, whereas the weights in the approximating branching are generated from the distributions of $\bar A$ and $\bar B$. Therefore, in order to prove Theorem \ref{main2} we  find  upper bounds  on the coupling error between $M\!P(\bar A_{(n)})$ and $MP(\bar A)$ and between $M\!P(\bar B_{(n)})$ and $MP(\bar B)$ which we state in Proposition \ref{claim2} and Lemma \ref{lemma0}.

Given the weights $A_1,\ldots, A_{n}$, for a coupling $\mathscr C$ of $A$ and $A_{(n)}$ 
we denote the corresponding probability measure and expectation by $P_{\mathscr C}$ and $E_{\mathscr C}$, respectively.  A  coupling of two random variables with  distributions $M\!P(\bar A)$ and $MP(\bar A_{(n)})$ can be constructed by first constructing a coupling $\mathscr C$ of their respective intensities $\bar A$ and $ A_{(n)}$, then generating a joint realization $(\bar A', \bar A_{(n)}')$ of these intensities according to $\mathscr C$ and in the next step using these intensities to generate two random variables from the Poisson distribution. 
If in the last step a maximal coupling is used then the (conditional) probability of a miscoupling is given by 
 $\tfrac{1}{2}d_{TV}(Po(\bar A'), Po(\bar A_{(n)}'))$. Here $d_{TV}$ denotes the total variation distance, i.e. for $a,b\in\mathbb{R}_{\geq 0}$ 
$$d_{TV}(Po(a), Po(b))=\sum_{k\geq 0}\abs{ \frac{a^ke^{-a}}{k!}-\frac{b^ke^{-b}}{k!}}.$$
 Thus, given the distribution of $A_{(n)}$ the probability of a miscoupling  is given by \\$E_{\mathscr C}(d_{TV}(Po(\bar A), Po(\bar A_{(n)})))/2$.

Let the coupling $\mathscr  C_n$  of $\bar A'$ and $\bar  A_{(n)}'$ be given by 
\begin{gather}\label{eq:min_coup}
\mathscr C_n:=\underset{\mathscr C}{\text{argmin}} E_{\mathscr C}\abs{\sqrt{\bar A}-\sqrt{\bar A_{(n)}}}
\end{gather}
for each $n\geq 1$, where the minimum extends over all couplings $\mathscr C$ of $A$ and $A_{(n)}$.

In the following proposition, $A_1,\ldots, A_n$ are random with respect to $P$. Thus, under $P$, $\mathscr C_n$ is a coupling of $\bar p$ and  the random probability measure $\bar p_n$.  
\begin{proposition}\label{claim2}
Assume that $E(A^q)<\infty$ for $q>\frac{3}{2}$. Let  $(\varepsilon_n)_{n\geq 1}$  be a sequence in $(0,\infty)$ such that if $q\neq 3$ then  $\varepsilon_n\log(n)\to \infty$ as $n\to 0$ and if $q= 3$ then  $\varepsilon_n\log(n)-\log(\log(n))\to \infty$ as $n\to 0$.  Then
 
$$P\left( E_{\mathscr C_n}(d_{TV}(Po(\bar A), Po(\bar A_{(n)})))\geq n^{-\gamma+\varepsilon_n}\right)\to 0\quad \text{as } n\to\infty $$

where $\gamma=\frac{1}{2}\wedge \frac{q-3/2}{q}$. 

\end{proposition}

\begin{remark}
Proposition \ref{claim2} holds also if $A$, $\bar A$ and $\bar A_{(n)}$ are replaced by $B$, $\bar B$ and $\bar B_{(n)}$.  
\end{remark}

\begin{proof}[Proof of Proposition \ref{claim2}]

We have that (the first inequality follows from \textcite[Theorem 1.C]{barbour1992poisson})

$$ E_{\mathscr C_n} \left( d_{TV}(Po(\bar A), Po(\bar A_{(n)}))\right)$$

$$\leq E_{\mathscr C_n}\p{\frac{1}{\sqrt{\bar A_{(n)}\vee \bar A}} \abs{{\bar A}-{\bar A_{(n)}}}}$$

$$\leq E_{\mathscr C_n}\p{\frac{1}{\frac{1}{2}\p{\sqrt{\bar A}+\sqrt {\bar A_{(n)}}}} \abs{\sqrt{\bar A}+\sqrt {\bar A_{(n)}}}\abs{\sqrt{\bar A}-\sqrt {\bar A_{(n)}}}}$$

$$= 2 E_{\mathscr C_n}\p{\abs{\sqrt{\bar A}-\sqrt {\bar A_{(n)}}}}.$$

Hence, by Proposition \ref{lemma:wass_dist} below,

$$E\p{ E_{\mathscr C_n} \left( d_{TV}(Po(\bar A), Po(\bar A_{(n)})\right);D_n}=\begin{cases}\OO(n^{-\frac{q-3/2}{q}}) & \text{ if }3/2<q<3\\
\OO(n^{-\frac{1}{2}}\log(n)) & \text{ if }q=3\\
\OO(n^{-\frac{1}{2}}) & \text{ if }q>3\\
\end{cases},$$
where $D_n$ is the event $\max_{i\in[n]}(A_i)>0$.
The assertion of the Proposition now follows from Markov's inequality and $P(D_n^c)=p(0)^n$, $p(0)<1$. 
\end{proof}

By the degree distributions of $G_n^{\text{aux}}$ given in  (\ref{eq_po}) and (\ref{eq_poB}), in order to arrive at  Theorem \ref{main2} we will also need a bound on the coupling error of  $M\!P\p{\bar B_{(n)}{\mu_A^{(n)}}/{\mu_A}}$
and $M\!P\p{\bar B_{(n)}}$ and of $M\!P\left( \bar A_{(n)} {\mu_B^{(n)}\floor{ n\mu_A/\mu_B}}/{n\mu_A}\right)$ and $M\!P\p{\bar A_{(n)}}$.

\begin{proposition}
[c.f. \textcite{Fournier2015}, Theorem 1]\label{lemma:wass_dist}
If $E(A^{q})<\infty$ for $q > 3/2$  then 
$$E\p{E_{\mathscr C_n}\p{\abs{\sqrt{\bar A}-\sqrt {\bar A_{(n)}}}};D_n}=\begin{cases}\OO(n^{-\frac{q-3/2}{q}}) & \text{ if }3/2<q<3\\
\OO(n^{-\frac{1}{2}}\log(n)) & \text{ if }q=3\\
\OO(n^{-\frac{1}{2}}) & \text{ if }q>3\\
\end{cases}
$$
where ${\mathscr C_n}$ is the coupling in (\ref{eq:min_coup}) and $D_n$ is the event that  $\max_{i\in[n]}(A_i)>0$. 

\end{proposition}

\begin{proof}[Proof of Proposition \ref{lemma:wass_dist}]
This proof is, in part, analogous to  the proof of Theorem 1 in \textcite{Fournier2015} and is presented here in full for completeness. 
The  differences between the present proof and the proof of Theorem 1 in \textcite{Fournier2015} arise mainly due to the size-biasing of the weights.

Throughout, $C_1, C_2,\ldots$ are positive constants that  depend only on $q$ and the distribution of $A$,  and $U\subset[0,\infty)$ is a generic Borel set. 

Note that $$E_{\mathscr C_n}\abs{\sqrt{\bar A}-\sqrt{\bar A_{(n)}}}$$ is the 1-Wasserstein  distance between the distributions of $\sqrt{\bar A}$ and $\sqrt{\bar A_{(n)}}$. Hence, with the notation $2^mF=\{2^mu:u\in F\}$ for any event $F$, we have by Lemma 5 and 6 in \textcite{Fournier2015}

\begin{equation}\label{eq:dtv}
    \begin{gathered}
E_{\mathscr C_n}\abs{\sqrt{\bar A}-\sqrt{\bar A_{(n)}}}\\
\\
\leq C_1 \sum_{m\geq 0} 2^m\sum_{\ell\geq 0}2^{-\ell} \sum_{F\in\mathcal P_\ell}\abs{\nu_n(2^m F\cap U_m)-\nu(2^m F\cap U_m)}
    \end{gathered}
\end{equation}

where $\mathcal P_\ell$ is the partition of [0,1] that consists of $\{0\}$ and $2^{-\ell+1}k+(0,2^{-\ell+1}]$ for $k\in\{0,1,\ldots, 2^{\ell-1}-1\}$, $U_0=[0,1]$ and $U_m=[0,2^m]\setminus [0, 2^{m-1}]$ for $m\geq 1$, and $\nu$ and $\nu_n$ are the probability distributions of $\sqrt{ \bar A}$ and $\sqrt{ \bar A_{(n)}}$, respectively. 
Now, with the notation $U^2=\{u^2:u\in U\}$, by the triangle inequality

\begin{equation}\label{eq:a21}
\begin{gathered}
\abs{\nu_n(U)-\nu(U)}\\
\\
=\abs{\bar p_n(U^2)-\bar p(U^2)}\\
\\
=\abs{\frac{\sum_{i=1}^nA_i\ind(A_i\in U^2)}{n\mu_A^{(n)}}-\frac{E(A\ind(A\in U^2))}{\mu_A}}\\
\\
\leq \frac{1}{n\mu_A^{(n)}}\sum_{i=1}^n{A_i\ind(A_i\in U^2})\abs{1-\frac{\mu_A^{(n)}}{\mu_A}}+ \frac{1}{\mu_A}\abs{{\frac{1}{n}\sum_{i=1}^nA_i\ind(A_i\in U^2)}-{E(A\ind(A\in U^2))}}.
\end{gathered}
\end{equation}
The second term in the right hand side in (\ref{eq:a21}) satisfies (with the first inequality following from Jensen's inequality)

\begin{equation}\label{eq:a2}
\begin{gathered}
E\p{\abs{{\frac{1}{n}\sum_{i=1}^nA_i\ind(A_i\in U^2)}-{E(A\ind(A\in U^2))}}}\\
\\
\leq \sqrt{Var\p{{\frac{1}{n}\sum_{i=1}^nA_i\ind(A_i\in U^2)}}}\\
\\
\leq(\sup U)^2\sqrt{\frac{1}{n}p(U^2)}
\end{gathered}
\end{equation}

and, again by the triangle inequality, 

\begin{equation}\label{eq:a3}
\begin{gathered}
E\p{\abs{\frac{1}{n}\sum_{i=1}^nA_i\ind(A_i\in U^2)-E(A\ind(A\in U^2))}}
\leq 2(\sup U)^2p(U^2).
\end{gathered}
\end{equation}

Combining (\ref{eq:a2}) and (\ref{eq:a3}) gives that the second term in the right hand side in (\ref{eq:a21}) satisfies 

\begin{equation}\label{eq:ab00}
    \begin{gathered}
    \frac{1}{\mu_A}\abs{{\frac{1}{n}\sum_{i=1}^nA_i\ind(A_i\in U^2)}-{E(A\ind(A\in U^2))}}\\
    \\
\leq C_2(\sup U)^2\p{\sqrt{\frac{1}{n}p(U^2)}\wedge p(U^2)}.
    \end{gathered}
\end{equation}

In order to find a similar upper bound on the  first  term in the right hand side in (\ref{eq:a21}) we define the event $S_n$ as

\begin{gather}\label{eq_sn}
S_n:=\{\mu_A^{{(n)}}\leq \mu_A-\kappa\}
\end{gather}

where $\kappa\in(0, \mu_A)$ is a fixed constant such that  $e^{(\mu_A-\kappa)}E(e^{-A})<1$. Then, with the second inequality following from  Hölder's inequality and the the last inequality from $\sum_{i=1}^n\ind(A_i\in U^2)\sim Bin (n, p(U^2))$, 

\begin{equation}\label{eq:ab12}
    \begin{gathered}
    E\p{ \frac{1}{n\mu_A^{(n)}}\sum_{i=1}^n{A_i\ind(A_i\in U^2})\abs{1-\frac{\mu_A^{(n)}}{\mu_A}};S_n^c\cap D_n}\\
    \\
    \leq\frac{(\sup U)^2}{(\mu_A-\kappa)}\p{\frac{1}{n} E\p{\sum_{i=1}^n{\ind(A_i\in U^2})\abs{1-\frac{\mu_A^{(n)}}{\mu_A}}}\wedge (2p(U^2))}\\
    \\
    \leq  \frac{(\sup U)^2}{(\mu_A-\kappa)}\p{ \sqrt{\frac{1}{n^3}E\p{\p{\sum_{i=1}^n{\ind(A_i\in U^2})}^2}Var\p{\frac{A}{\mu_A}}}\wedge (2p(U^2))}\\
    \\
    \leq \frac{(\sup U)^2}{(\mu_A-\kappa)}\p{ \sqrt{\frac{1}{n^3}np(U^2)(np(U^2)+1)Var\p{\frac{A}{\mu_A}}}\wedge (2p(U^2))}\\
    \\
    \leq C_3(\sup U)^2\p{\sqrt{\frac{1}{ n}p(U^2)}\wedge p(U^2))}.
    \end{gathered}
\end{equation}

By the  inequalities in (\ref{eq:ab00}) and (\ref{eq:ab12}) together with the fact that for any  $G\subset [0,\infty)$ it holds that $\sup(G\cap U_m^2)\leq \sup(U_m^2)= 2^{2m}$, we have

\begin{equation}\label{eq:dtv1}
    \begin{gathered}
    \sum_{F\in\mathcal P_\ell}E\p{\abs{\nu_n(2^m F\cap U_m)-\nu(2^m F\cap U_m)}; D_n\cap S_n^c}\\
    \\
    \leq C_4\p{ 2^{2m}\sum_{F\in\mathcal P_\ell}\p{\sqrt{\frac{1}{n}p(2^{2m} F^2\cap U_m^2)}\wedge p(2^{2m} F^2\cap U_m^2)}}
    \end{gathered}
\end{equation}

and  (with the second inequality following from  the Cauchy–Schwarz inequality  and the fact that, since $E(A^q)$ is  finite, $P(U_m^2)\leq E(A^q)2^{-2qm}$)

\begin{equation}\label{eq:dtv2}
    \begin{gathered}
     \sum_{F\in\mathcal P_\ell}\p{\sqrt{\frac{1}{n}p(2^{2m} F^2\cap U_m^2)}\wedge p(2^{2m} F^2\cap U_m^2)}\\
     \\
     \leq \p{\frac{1}{\sqrt n}\sum_{F\in\mathcal P_\ell}\sqrt{p(2^{2m} F^2\cap U_m^2)}}\wedge p(U_m^2)\\
     \\
     \leq C_5\p{\p{\frac{1}{\sqrt n}\sqrt{\abs{P_\ell}p(  U_m^2)}
     }\wedge 2^{-2qm}}\\
     \\
     \leq C_5\p{\p{\frac{1}{\sqrt n}\sqrt{2^\ell E(A^q)2^{-2qm}}
     }\wedge 2^{-2qm}}
     \\
     \\
     \leq C_6\p{{\frac{1}{\sqrt n}2^{\ell/2-qm}
     }\wedge 2^{-2qm}}. 
    \end{gathered}
\end{equation}

Hence, by (\ref{eq:dtv}), (\ref{eq:dtv1}) and (\ref{eq:dtv2}),

$$E\p{E_{\mathscr C_n}\p{\abs{\sqrt{\bar A}-\sqrt{\bar A_{(n)}}}};S_n^c\cap D_n}$$

$$\leq  C_7 \sum_{m\geq 0} 2^m\sum_{\ell\geq 0}2^{-\ell+2m}\p{{\frac{1}{\sqrt n}2^{\ell/2-qm}
     }\wedge 2^{-2qm}}
$$

$$\leq   C_7 \sum_{m\geq 0} 2^{3m}\sum_{\ell\geq 0}2^{-\ell/2}\p{{\frac{1}{\sqrt n}2^{-qm}
     }\wedge 2^{-2qm}}
$$

$$\leq   C_8 \sum_{m\geq 0} \p{{\frac{1}{\sqrt n}2^{m(3-q)}
     }\wedge 2^{-m(2q-3)}}.
$$

If $q>3$ then 

$$     C_8 \sum_{m\geq 0} \p{{\frac{1}{\sqrt n}2^{m(3-q)}
     }\wedge 2^{-m(2q-3)}} \leq C_9  \frac{1}{\sqrt n}.
$$

If $q\in(3/2, 3)$ then with $m_n=\floor{\log (n )/(2q\log( 2))}$

$$     C_8 \sum_{m\geq 0} \p{{\frac{1}{\sqrt n}2^{m(3-q)}
     }\wedge 2^{-m(2q-3)}} $$

$$\leq  C_8 \sum_{m=0}^{m_n} {\frac{1}{\sqrt n}2^{-m(q-3)}
     }+ C_8 \sum_{m>m_n}^\infty 2^{-m(2q-3)}
$$

$$=O\p{n^{-\p{1-\frac{3}{2q}}}}.$$

If $q=3$ then with  $a_n=\floor{\log(n)/log(2)}$

$$     C_8 \sum_{m\geq 0} \p{{\frac{1}{\sqrt n}2^{m(3-q)}
     }\wedge 2^{-m(2q-3)}} $$

$$\leq  C_8 {\frac{ a_n}{\sqrt n}
     }+ C_8 \sum_{m\geq a_n}^\infty 2^{-m(2q-3)}
$$

$$=O\p{n^{-1/2}\log(n)}.$$

Thus it only remains to bound the expectation of $E_{\mathscr C_n}\p{\abs{\sqrt{\bar A}-\sqrt{\bar A_{(n)}}}}$ on  $S_n\cap D_n$, where 
 $S_n$ is the event in  (\ref{eq_sn}). Now, with the third inequality following from the fact that on $S_n$ we have  $\sqrt{A_i}\leq \sqrt{n(\mu_A-\kappa)}$  for $i=1,\ldots, n$,  

$$E\p{E_{\mathscr C_n}\p{\abs{\sqrt{\bar A}-\sqrt{\bar A_{(n)}}}};  S_n\cap D_n}$$

$$\leq P(S_n\cap D_n)E(\sqrt{\bar A})+E\p{\sqrt{\bar A_{(n)}};S_n\cap D_n}$$

$$= P(S_n\cap D_n)E(\sqrt{\bar A})+E\p{\frac{\sum_{i=1}^nA_i^{\frac{3}{2}}}{\sum_{i=1}^nA_i};S_n\cap D_n}$$

$$\leq P(S_n\cap D_n)E(\sqrt{\bar A})+E\p{{\sum_{i=1}^n\sqrt {A_i}};S_n\cap D_n}$$

$$\leq P(S_n\cap D_n)E(\sqrt{\bar A})+E\p{n^{3/2}\sqrt{\mu_A-\kappa};S_n\cap D_n}$$

$$=\OO(n^{3/2}P(S_n))$$

$$=\OO(n^{3/2}e^{n(\mu_a-\kappa)}E(e^{-A})^n)$$

where the last step follows from the Chernoff  bound $P(S_n)\leq e^{n(\mu_a-\kappa)}E(e^{-A})^n$. The assertion now follows by recalling that $e^{(\mu_a-\kappa)}E(e^{-A})<1$.

\end{proof}

\begin{lemma}\label{lemma0}
Let $\varepsilon_n$ be as in Proposition \ref{claim2} and $E(A^2), E(B^2)<\infty$. Then, w.h.p.,  
\begin{gather}\label{lemmaeq1}
d_{TV}\p{ Po\p{\bar B_{(n)}{\mu_A^{(n)}}/{\mu_A}},Po\p{\bar B_{(n)}}}\leq n^{-\frac{1}{2}+\varepsilon_n}
\end{gather}
and 
\begin{gather}\label{lemmaeq2}
 d_{TV}\p{ Po\p{ \bar A{\mu_B^{(n)}\floor{ n\mu_A/\mu_B}}/{n\mu_A}},Po\p{\bar A_{(n)}}}\leq n^{-\frac{1}{2}+\varepsilon_n}.
 \end{gather}

\end{lemma}

\begin{proof}[Proof of Lemma \ref{lemma0}]

Define $$H_n:=\left\{\abs{1-{\mu_A^{(n)}}/{\mu_A}}\leq n^{-\frac{1}{2}+\frac{1}{2}\varepsilon_n} \right\}. $$
By Chebyshev's inequality 
$P\p{H_n^c}=\OO\p{n^{-\varepsilon_n}}.$ Now (again by \textcite[Theorem 1.C]{barbour1992poisson})

$$E\p{d_{TV}\p{Po\p{\bar B_{(n)}{\mu_A^{(n)}}/{\mu_A}} , Po\p{\bar B_{(n)}}}; H_n}$$

$$\leq E\p{ \abs{ \bar B_{(n)}{\mu_A^{(n)}}/{\mu_A}-\bar B_{(n)}}; H_n}$$

$$\leq E\p{ \bar B_{(n)} n^{-\frac{1}{2}+\frac{1}{2}\varepsilon_n}}=\OO(n^{-\frac{1}{2}+\frac{1}{2}\varepsilon_n}).$$

This implies, using the union bound and Markov's inequality, 

$$P\p{d_{TV}\p{Po\p{\bar B_{(n)}{\mu_A^{(n)}}/{\mu_A}} , Po\p{\bar B_{(n)}}} \geq n^{-\frac{1}{2}+\varepsilon_n} }$$

$$\leq  P(H_n^c) +\OO(n^{-\frac{1}{2}\varepsilon_n}).$$

This proves the inequality in (\ref{lemmaeq1}), and the proof of the inequality in  (\ref{lemmaeq2}) is completely analogous. 

\end{proof}

\begin{proof}[Proof of Theorem \ref{main2}]
Let $\varepsilon_n, q$ and $\gamma$ be as in Theorem  \ref{main2}, and let $T\in\mathbb{N}_0$ be the number of iterations in  the construction algorithm on page \pageref{constralg}  (i.e.  $\mathcal N_T=\mathcal E_T $). By Claim  \ref{claim:coupling}, w.h.p. the vertices of $V_n$ and $V_n'$ that are explored in the first $\floor{n^{\gamma-\varepsilon_n}}\wedge T$ steps of the  construction algorithm are distinct. 
Combining Claim  \ref{claim:coupling} with Proposition \ref{claim2} gives the assertion of Theorem \ref{main2} by the union bound.
\end{proof}

\section{Acknowledgements}

This research is supported by the Swedish Research Council (Vetenskapsr{\aa}det) grant 2016-04566. 

 The author would also like to thank Pieter Trapman for  fruitful discussions and constructive criticism of the paper.

\newpage

\printbibliography
\end{document}